\documentclass[a4paper,11pt,reqno]{amsart}
\usepackage{amssymb}
\usepackage{amsmath,enumerate}
\usepackage{graphicx}
\usepackage{tikz}

\usepackage[top=30truemm,bottom=30truemm,left=20truemm,right=20truemm]{geometry}
\allowdisplaybreaks[4]

\newcommand{\RWpictureA}{
  \begin{tikzpicture}
    \coordinate (t1) at (0,0) node at (t1) {$\times$} node at (t1) [below] {$t_1$};
    \node (t1+1) at (6,0) {$\times$} node at (t1+1) [below] {$t_1\!+\!1$};
    \node (t1+t) at (1,4) {$\times$} node at (t1+t) [above] {$t_1\!+\!\tau$};
    \node (t1+1t) at (7,4) {$\times$} node at (t1+1t) [above] {$t_1\!+\!1\!+\!\tau$};
    \coordinate (P0) at (1,0) node at (P0) [below right] {$P_0$};
    \coordinate (P1) at (0.25,1) node at (P1) [above left] {$P_1$};
    \coordinate (P2) at (-1,0) node at (P2) [left] {$P_2$};
    \coordinate (P3) at (-0.25,-1) node at (P3) [below] {$P_3$};
    \fill (P0) circle (0.1) (P1) circle (0.1) (P2) circle (0.1) (P3) circle (0.1);
    \coordinate (P01) at (7,0) coordinate (P0t) at (2,4) coordinate (P01t) at (8,4)
      coordinate (P11) at (6.25,1) coordinate (P1t) at (1.25,5) coordinate (P11t) at (7.25,5)
      coordinate (P21) at (5,0) coordinate (P2t) at (0,4) coordinate (P21t) at (6,4)
      coordinate (P31) at (5.75,-1) coordinate (P3t) at (0.75,3) coordinate (P31t) at (6.75,3);
    \fill (P01) circle (0.07) (P0t) circle (0.07) (P01t) circle (0.07)
      (P11) circle (0.07) (P1t) circle (0.07) (P11t) circle (0.07)
      (P21) circle (0.07) (P2t) circle (0.07) (P21t) circle (0.07)
      (P31) circle (0.07) (P3t) circle (0.07) (P31t) circle (0.07);
    \draw (P0) -- (P21) node [pos=0.5] {$>$} 
      (P1) -- (P3t) node [pos=0.5,sloped] {$>$}
      (P0t) -- (P21t) node [pos=0.5] {$>$}
      (P11) -- (P31t) node [pos=0.5,sloped] {$>$};
    \draw [dotted] (t1) -- (2,1.5) node (t2) {$\times$} 
      -- (3,2.1) node (t3) {$\times$} -- (3.5,2.4) node (t4) {$\times$} -- (4.5,3) node (tn) {$\times$};
    \node at (t2) [below] {$t_2$} node at (tn) [below] {$t_n$};
    \draw (P0) arc (0:65:1) node [pos=0.5,sloped] {$<$};
    \draw (P1) arc (76:180:1) node [pos=0.5,sloped] {$<$};
    \draw (P2) arc (180:256:1) node [pos=0.5,sloped] {$>$};
    \draw (P3) arc (256:360:1) node [pos=0.5,sloped] {$>$};
    \draw (P01) arc (0:65:1) node [pos=0.5,sloped] {$<$}
        (P11) arc (76:180:1) node [pos=0.5,sloped] {$<$}
        (P21) arc (180:256:1) node [pos=0.5,sloped] {$>$}
        (P31) arc (256:360:1) node [pos=0.5,sloped] {$>$}
      (P0t) arc (0:65:1) node [pos=0.5,sloped] {$<$}
        (P1t) arc (76:180:1) node [pos=0.5,sloped] {$<$}
        (P2t) arc (180:256:1) node [pos=0.5,sloped] {$>$}
        (P3t) arc (256:360:1) node [pos=0.5,sloped] {$>$}
      (P01t) arc (0:65:1) node [pos=0.5,sloped] {$<$}
        (P11t) arc (76:180:1) node [pos=0.5,sloped] {$<$}
        (P21t) arc (180:256:1) node [pos=0.5,sloped] {$>$}
        (P31t) arc (256:360:1) node [pos=0.5,sloped] {$>$};
    \coordinate (P02) at (2.5,1.5) node at (P02) [right] {$P_0^{(2)}$}; 
    \coordinate (P0n) at (5,3) node at (P0n) [right] {$P_0^{(n)}$}; 
    \draw (P02) arc (0:345:0.5) node [pos=0.25,sloped] {$<$}; 
    \draw (P0n) arc (0:345:0.5) node [pos=0.25,sloped] {$<$}; 
    \fill (P02) circle (0.07) (P0n) circle (0.07);
    \draw (P0) to [out=30, in=290] node [pos=0.5,sloped] {$>$} (P02)
      (P0) to [out=10, in=290] node [pos=0.5,sloped] {$>$} (P0n);
    \draw (3.5,0) node [below] {$l_0$}
      (0.5,2.3) node [left] {$l_{\infty}$}
      (1,0.8) node {$m_0$} (7,0.8) node {$m_0$} (2,4.8) node {$m_0$} (8,4.8) node {$m_0$} 
      (-1.1,0.8) node {$m_1$} (4.9,0.8) node {$m_1$} (-0.1,4.8) node {$m_1$} (5.9,4.8) node {$m_1$} 
      (-1,-0.9) node {$m_2$} (5,-0.9) node {$m_2$} (0,3.1) node {$m_2$} (5.8,3.5) node {$m_2$} 
      (1,-0.9) node {$m_3$} (7.1,-0.8) node {$m_3$} (2.1,3.2) node {$m_3$} (8.1,3.2) node {$m_3$} 
      (1.6,0.6) node {$l_2$} (4.2,1.5) node {$l_n$}
      (1.8,2) node [above] {$s_2$} (4.2,3.4) node [left] {$s_n$};
  \end{tikzpicture}
}

\newcommand{\RWpictureB}{
  \begin{tikzpicture}
    \draw (0,0) -- (6,0) -- (7,4) -- (1,4) -- cycle;
    \draw [dotted] (1,0.7) node (t1) {$\times$} -- (1.5,1) node (t2) {$\times$} 
      -- (3,2) node (tj) {$\times$} -- (4.5,2.6) node (tn-1) {$\times$} -- (5.5,3) node (tn) {$\times$};
    \draw (t1) node [below] {$t_1$} (tj) node [below] {$t_j$} (tn) node [below] {$t_n$};
    \coordinate (P0) at (4,2) node at (P0) [below right] {$P_0^{(j)}$}
      coordinate (P1) at (3.25,3) node at (P1) [below] {$P_1^{(j)}~$}
      coordinate (P2) at (2,2) node at (P2) [above left] {$P_2^{(j)}$}
      coordinate (P3) at (2.75,1) node at (P3) [below right] {$P_3^{(j)}$};
    \fill (P0) circle (0.05) (P1) circle (0.05) (P2) circle (0.05) (P3) circle (0.05);
    \draw (P0) arc (0:65:1) node [pos=0.5,sloped] {$<$}
      (P1) arc (76:180:1) node [pos=0.5,sloped] {$<$}
      (P2) arc (180:256:1) node [pos=0.5,sloped] {$>$}
      (P3) arc (256:360:1) node [pos=0.5,sloped] {$>$};
    \draw[dashed] (P0) -- (6.5,2) node [pos=0.7,sloped] {$>$}
      (P2) -- (0.5,2) node [pos=0.75,sloped] {$>$};
    \draw[dash dot] (P1) -- (3.5,4) node [pos=0.5,sloped] {$>$}
      (P3) -- (2.5,0) node [pos=0.7,sloped] {$>$};
    \draw (5.8,2) node [below] {$l_0^{(j)}$}
      (4.2,3.7) node [left] {$l_{\infty}^{(j)}$}
      (4,3) node {$m_0^{(j)}$} (2,3) node {$m_1^{(j)}$} (2.1,0.9) node {$m_2^{(j)}$} (4,1) node {$m_3^{(j)}$};
    \coordinate (P0r) at (9,3.5) node at (P0r) [left] {$P_0^{(j)}$}
      coordinate (P1r) at (10,0.5) node at (P1r) [left] {$P_1^{(j)}$}
      coordinate (P2r) at (11,3.5) node at (P2r) [right] {$P_2^{(j)}$}
      coordinate (P3r) at (11,2.5) node at (P3r) [right] {$P_3^{(j)}$};
    \fill (P0r) circle (0.05) (P1r) circle (0.05) (P2r) circle (0.05) (P3r) circle (0.05);
    \draw[dashed] (P0r) -- (P2r) node [pos=0.5,sloped] {$>$};
    \draw[dash dot] (P1r) -- (P3r) node [pos=0.5,sloped] {$>$};
    \draw (10,3) node {$l_0^{(j)}$} (11,1.5) node {$l_{\infty}^{(j)}$};
  \end{tikzpicture}
}

\newcommand{\RWpictureC}{
  \begin{tikzpicture}
    \draw (0,0) -- (6,0) -- (7,4) -- (1,4) -- cycle;
    \draw [dotted] (1,0.7) node (t1) {$\times$} -- (1.7,1.1) node (t2) {$\times$} 
      -- (3,2) node (tj) {$\times$} -- (4.5,2.6) node (tn-1) {$\times$} -- (5.5,3) node (tn) {$\times$};
    \draw (t1) node [below] {$t_1$} (tj) node [below] {$t_j$} (tn) node [below] {$t_n$};
    \coordinate (P0j) at (3.4,2.4)
      coordinate (P0n) at (6,3) node at (P0n) [above right] {$P_0^{(n)}$};
    \fill (P0j) circle (0.05) (P0n) circle (0.05);
    \draw (P0j) arc (45:390:0.55) node [pos=0.25,sloped] {$<$}
      (P0n) arc (0:345:0.5) node [pos=0.25,sloped] {$<$}; 
    \coordinate (ljn) at (5.5,3.8);
    \draw (P0n) to [out=80, in=0] (ljn)
      (ljn) to [out=180, in=40] node [pos=0.5,sloped] {$<$} (P0j);
    \draw (4,3.4) node {$l_j^{(n)}$}
      (5.9,2.3) node {$s_n$}
      (4.2,1.2) node {$e^{2\pi \sqrt{-1}(c_{j+1}+\cdots +c_n)}s_j$};
  \end{tikzpicture}
}

\newcommand{\RWpictureD}{
  \begin{tikzpicture}
    \draw (0,0) -- (6,0) -- (7,4) -- (1,4) -- cycle;
    \coordinate (t1) at (1,0.7) node at (t1) {$\times$}
      coordinate (t2) at (1.7,1.1) node at (t2) {$\times$}
      coordinate (tp) at (2.4,1.5) 
      coordinate (tq-1) at (3,2) node at (tq-1) {$\times$}
      coordinate (tq) at (3.7,2.3) node at (tq) {$\times$}
      coordinate (tn-1) at (4.5,2.6) node at (tn-1) {$\times$}
      coordinate (tn) at (5.5,3) node at (tn) {$\times$};
    \fill (tp) circle (0.05);
    \draw (t1) node [below] {$t_1^{\circ}$} (tp) node [below] {$t_p^{\circ}$} 
      (tq) node [below] {$t_q^{\circ}$} (tn) node [below] {$t_n^{\circ}$};
    \coordinate (l1) at (4,1.5) 
      coordinate (l2) at (4,2.6)
      coordinate (l3) at (3.5,2.6)
      coordinate (l4) at (3.3,1.9); 
    \draw (tp) to [out=0, in=200] (l1)
      (l1) to [out=20, in=300] node [pos=0.4,sloped] {$>$} (l2)
      (l2) to [out=120, in=60] (l3)
      (l3) to [out=240, in=90] node [pos=0.2,sloped] {$<$} (l4)
      (l4) to [out=270, in=30] (tp);
    \draw (l1) node [below right] {$\ell_{pq}$};
  \end{tikzpicture}
}

\newcommand{\RWpictureE}{
  \begin{tikzpicture}
    \draw (0,0) -- (6,0) -- (7,4) -- (1,4) -- cycle;
    \draw [dotted] (1.5,1) node (t1) {$\times$} -- (3,2) node (t2) {$\times$} 
      -- (4,2.4) node (tj) {$\times$} -- (4.75,2.7) node (tn-1) {$\times$} -- (5.5,3) node (tn) {$\times$};
    \draw (t1) node [below] {$t_1$} (t2) node [below] {$t_2$} (tn) node [below] {$t_n$};
    \coordinate (P01) at (3.6,2) coordinate (P21) at (2.4,2) 
      coordinate (11) at (3,1) coordinate (12) at (1.5,1.3) coordinate (13) at (1,1)
      coordinate (14) at (1.5,0.5) coordinate (15) at (3,0.8) coordinate (16) at (5,2);
    \draw (P01) arc (0:360:0.6) node [pos=0.35,sloped] {$<$}
      (P01) to [out=0, in=0] (11)
      (11) to [out=180, in=0] node [pos=0.5,sloped] {$<$} (12)
      (12) to [out=180, in=90] (13)
      (13) to [out=270, in=180] (14)
      (14) to [out=0, in=180] node [pos=0.3,sloped] {$>$} (15)
      (15) to [out=0, in=180] (16)
      (16) -- (6.5,2) node [pos=0.5,sloped] {$>$} 
      (0.5,2) -- (P21) node [pos=0.4,sloped] {$>$};
    \coordinate (P12) at (3.125,2.5) coordinate (P32) at (2.875,1.5)
      coordinate (21) at (2.625,0.5) coordinate (22) at (2.3,0.4) coordinate (23) at (1.5,0.25)
      coordinate (24) at (0.7,1) coordinate (25) at (1.5,1.5);
    \draw[dashed] (P12) arc (76:436:0.5) node [pos=0.7,sloped] {$>$}
      (P12) -- (3.5,4) node [pos=0.5,sloped] {$>$} 
      (2.5,0) -- (21) node [pos=0.5,sloped] {$>$} 
      (21) to [out=90, in=60] (22)
      (22) to [out=240, in=0] (23)
      (23) to [out=180, in=270] (24)
      (24) to [out=90, in=180] node [pos=0.5,sloped] {$>$} (25)
      (25) to [out=0, in=270] (P32);
    \draw (5,1.6) node {$\gamma_{20}^{(1n)}$}
      (2.8,3.4) node {$\gamma_{2\infty}^{(1n)}$};
  \end{tikzpicture}
}

\newcommand{\RWpictureF}{
  \begin{tikzpicture}
    \coordinate (0) at (0,0) coordinate (1) at (4.5,0) coordinate (t) at (0.75,3) coordinate (1+t) at (5.25,3)
      coordinate (2) at (9,0) coordinate (2+t) at (9.75,3) coordinate (2t) at (1.5,6) coordinate (1+2t) at (6,6);
    \draw (0) -- (1) -- (1+t) -- (t) -- cycle;
    \draw [dotted] (1) -- (2) -- (2+t) -- (1+t) 
      (t) -- (2t) -- (1+2t) -- (1+t);
    \coordinate (t1) at (1,0.6) node at (t1) {$\times$}
      coordinate (t2) at (1.75,1.05) node at (t2) {$\times$}
      coordinate (tp) at (2.5,1.5) 
      coordinate (tn-1) at (3.25,1.95) node at (tn-1) {$\times$}
      coordinate (tn) at (4,2.4) node at (tn) {$\times$}
      coordinate (t1+1) at (5.5,0.6) node at (t1+1) {$\times$}
      coordinate (t2+1) at (6.25,1.05) node at (t2+1) {$\times$}
      coordinate (tp+1) at (7,1.5) 
      coordinate (tn-1+1) at (7.75,1.95) node at (tn-1+1) {$\times$}
      coordinate (tn+1) at (8.5,2.4) node at (tn+1) {$\times$}
      coordinate (t1+t) at (1.75,3.6) node at (t1+t) {$\times$}
      coordinate (t2+t) at (2.5,4.05) node at (t2+t) {$\times$}
      coordinate (tp+t) at (3.25,4.5) 
      coordinate (tn-1+t) at (4,4.95) node at (tn-1+t) {$\times$}
      coordinate (tn+t) at (4.75,5.4) node at (tn+t) {$\times$};
    \fill (tp) circle (0.05) (tp+1) circle (0.05) (tp+t) circle (0.05);
    \draw (t1) node [below] {$t_1^{\circ}$} (tp) node [below] {$t_p^{\circ}$} (tn) node [below] {$t_n^{\circ}$}
      (t1+1) node [below right] {$t_1^{\circ}+1$} (tp+1) node [below right] {$t_p^{\circ}+1$} 
      (tn+1) node [below right] {$t_n^{\circ}+1$}
      (t1+t) node [above left] {$t_1^{\circ}+\tau$} (tp+t) node [above] {$t_p^{\circ}+\tau~~$} 
      (tn+t) node [above] {$t_n^{\circ}+\tau$};
    \draw[thick] (tp) -- (tp+1) node [pos=0.4,sloped] {$>$};
    \draw[thick] (tp) -- (tp+t) node [pos=0.4,sloped] {$>$};
    \draw (4,1.2) node {$\ell_{p0}$}
      (2.3,2.4) node {$\ell_{p\infty}$};
  \end{tikzpicture}
}

\newcommand{\RWpictureH}{
  \begin{tikzpicture}
    \draw (0,0) -- (6,0) -- (7,4) -- (1,4) -- cycle;
    \draw [dotted] (1,1) node (t1) {$\times$} -- (2.5,1.5) node (t2) {$\times$} 
      -- (4,2) node (t3) {$\times$}  -- (6,3) node (tn) {$\times$};
    \draw (t1) node [below] {$t_1$} (t2) node [below] {$t_2$} (t3) node [below] {$t_3$} (tn) node [below] {$t_n$};
    \coordinate (P10) at (1.5,1) coordinate (P20) at (3,1) coordinate (12) at (2.5,0.5);
    \draw[thick] (P10) arc (0:360:0.5) node [pos=0.3,sloped] {$<$}
      (P20) arc (-45:315:0.7) node [pos=0.3,sloped] {$<$}
      (P10) to [out=0, in=180] node [pos=0.8,sloped] {$>$} (12)
      (12) to [out=0, in=315] (P20);
    \coordinate (P30) at (4.5,2) coordinate (P20l) at (3,1.5) 
      coordinate (231) at (3.2,1.5) coordinate (232) at (4.2,1.3);
    \draw (P30) arc (0:360:0.5) node [pos=0.3,sloped] {$<$}
      (P20l) arc (0:360:0.5) node [pos=0.4,sloped] {$<$}
      (P20l) to [out=0, in=180] (231)
      (231) to [out=0, in=180] node [pos=0.5,sloped] {$>$} (232)
      (232) to [out=0, in=330] (P30);
    \fill (231) circle (0.07);
    \draw (3,0.4) node {$\gamma_{12}$}
      (4.7,1.3) node {$\gamma_{23}^{\vee}$};
  \end{tikzpicture}
}

\newcommand{\RWpictureI}{
  \begin{tikzpicture}
    \draw (0,0) -- (6,0) -- (7,4) -- (1,4) -- cycle;
    \draw [dotted] (1.5,1.5) node (t1) {$\times$} -- (3,2) node (t2) {$\times$} 
      -- (4,2.4) node (tj) {$\times$} -- (5,2.8) node (tn-1) {$\times$} -- (6,3.2) node (tn) {$\times$};
    \draw (t1) node [below] {$t_1$} (tj) node [below] {$t_j$} (tn) node [below] {$t_n$};
    \coordinate (P01) at (2,1.5) coordinate (P0j) at (4.5,2.4) 
      coordinate (1j) at (4,1.5) 
      coordinate (P012) at (2.2,1.5); 
    \draw[thick] (P01) arc (0:360:0.5) node [pos=0.3,sloped] {$<$}
      (P0j) arc (0:360:0.5) node [pos=0.3,sloped] {$<$} 
      (P01) to [out=0, in=180] (P012)
      (P012) to [out=0, in=180] node [pos=0.5,sloped] {$>$} (1j)
      (1j) to [out=0, in=330] (P0j);
    \draw (P012) arc (0:360:0.7) node [pos=0.5,sloped] {$>$}
      (1.67,2.18) -- (2.125,4) node [pos=0.3,sloped] {$>$}
      (1.125,0) -- (1.33,0.82) node [pos=0.5,sloped] {$>$}; 
    \fill (P012) circle (0.1);
    \draw (4,1.2) node {$\gamma_{1j}$}
      (2.4,3.2) node {$\gamma_{1\infty}^{\vee}$};
  \end{tikzpicture}
}

\newcommand{\RWpictureJ}{
  \begin{tikzpicture}
    \draw (0,0) -- (6,0) -- (7,4) -- (1,4) -- cycle;
    \draw [dotted] (2,1.5) node (t1) {$\times$}
      -- (4,2.4) node (tj) {$\times$} -- (5,2.8) node (tn-1) {$\times$} -- (6,3.2) node (tn) {$\times$};
    \draw (t1) node [below] {$t_1$}  (tn) node [below] {$t_n$};
    \coordinate (P0) at (2.7,1.5) coordinate (P2) at (1.3,1.5) coordinate (P1) at (2.17,2.18)
      coordinate (P12) at (2.12,1.985) coordinate (P32) at (1.88,1.015)
      coordinate (int1) at (2.35,2.1) coordinate (int2) at (2,0.8);
    \draw[thick] (P1) arc (76:432:0.7) node [pos=0.4,sloped] {$>$}
      (P0) -- (6.375,1.5) node [pos=0.5,sloped] {$>$}
      (0.375,1.5) -- (P2) node [pos=0.4,sloped] {$>$};
    \draw (P12) arc (76:420:0.5) node [pos=0.2,sloped] {$<$}
      (P12) to [out=20, in=270] (int1)
      (int1) -- (2.8,4) node [pos=0.4,sloped] {$>$}
      (1.8,0) -- (int2) node [pos=0.4,sloped] {$>$}
      (int2) to [out=90, in=256] (P32);
    \fill (int1) circle (0.1);
    \draw [fill=white] (int2) circle (0.1);
    \draw (4,1.2) node {$\gamma_{10}$}
      (2.2,3.2) node {$\gamma_{1\infty}^{\vee}$};
  \end{tikzpicture}
}

\newcommand{\RWpictureK}{
  \begin{tikzpicture}
    \draw (0,0) -- (6,0) -- (7,4) -- (1,4) -- cycle;
    \draw [dotted] (1,0.7) node (t1) {$\times$} -- (1.7,1.1) node (t2) {$\times$} 
      -- (3,2) node (tj) {$\times$} -- (4.5,3) node (tn-1) {$\times$} -- (5.5,3.4) node (tn) {$\times$};
    \draw (t1) node [below] {$t_1$} (tj) node [below] {$t_j$} (tn) node [right] {$t_n$};
    \coordinate (P01) at (3.7,2) coordinate (P21) at (2.3,2) 
      coordinate (P02) at (3.5,2) coordinate (P22) at (2.5,2) 
      coordinate (int1) at (3.66,2.24) coordinate (int2) at (2.34,2.24);
    \draw[thick] (P01) arc (0:360:0.7) node [pos=0.3,sloped] {$<$}
      (P01) -- (6.5,2) node [pos=0.5,sloped] {$>$}
      (0.5,2) -- (P21) node [pos=0.4,sloped] {$>$};
    \draw (P02) arc (0:360:0.5) node [pos=0.2,sloped] {$<$}
      (P02) to [out=40, in=210] (int1)
      (int1) to [out=30, in=180] node [pos=0.4,sloped] {$>$} (6.625,2.5)
      (0.625,2.5) to [out=0, in=150] node [pos=0.5,sloped] {$>$} (int2)
      (int2) to [out=330, in=140] (P22);
    \fill (int1) circle (0.1);
    \draw [fill=white] (int2) circle (0.1);
    \draw (5,1.7) node {$\gamma_{10}$}
      (5.7,2.8) node {$\gamma_{10}^{\vee}$};
  \end{tikzpicture}
}

\newcommand{\RWpictureL}{
  \begin{tikzpicture}
    \draw (0,0) -- (6,0) -- (7,4) -- (1,4) -- cycle;
    \draw [dotted] (1,0.7) node (t1) {$\times$} -- (1.7,1.1) node (t2) {$\times$} 
      -- (3,2) node (tj) {$\times$} -- (4.5,2.6) node (tn-1) {$\times$} -- (5.5,3) node (tn) {$\times$};
    \draw (t1) node [below] {$t_1$} (tj) node [below] {$t_j$} (tn) node [below] {$t_n$};
    \coordinate (P01) at (3.7,2) coordinate (P11) at (3.17,2.68) coordinate (P31) at (2.83,1.32) 
      coordinate (P02) at (3.5,2)  coordinate (P12) at (3.12,2.485) coordinate (P32) at (2.88,1.515) 
      coordinate (int1) at (3.35,2.6) coordinate (int2) at (3,1.3);
    \draw[thick] (P01) arc (0:360:0.7) node [pos=0,sloped] {$>$}
      (P11) -- (3.5,4) node [pos=0.4,sloped] {$>$}
      (2.5,0) -- (P31) node [pos=0.4,sloped] {$>$};
    \draw (P02) arc (0:360:0.5) node [pos=0.4,sloped] {$<$}
      (P12) to [out=60, in=240] (int1)
      (int1) to [out=60, in=256] node [pos=0.4,sloped] {$>$} (4,4)
      (3,0) to [out=76, in=300] node [pos=0.5,sloped] {$>$} (int2)
      (int2) to [out=120, in=300] (P32);
    \fill (int1) circle (0.1);
    \draw [fill=white] (int2) circle (0.1);
    \draw (3,3.5) node {$\gamma_{1\infty}$}
      (4.3,3.5) node {$\gamma_{1\infty}^{\vee}$};
  \end{tikzpicture}
}

\newtheorem{Def}{Definition}[section]
\theoremstyle{remark}
\newtheorem{Rem}[Def]{Remark}
\newtheorem{Ex}[Def]{Example}

\newtheorem*{Ack}{Acknowledgments}
\theoremstyle{plain}
\newtheorem{Th}[Def]{Theorem}
\newtheorem{Prop}[Def]{Proposition}
\newtheorem{Lem}[Def]{Lemma}
\newtheorem{Cor}[Def]{Corollary}
\newtheorem{Fact}[Def]{Fact}

\newcommand{\R}{\mathbb{R}}
\newcommand{\Z}{\mathbb{Z}}

\newcommand{\C}{\mathbb{C}}
\renewcommand{\P}{\mathbb{P}}
\renewcommand{\H}{\mathbb{H}}

\newcommand{\CE}{\mathcal{E}}

\newcommand{\CL}{\mathcal{L}}

\newcommand{\CO}{\mathcal{O}}
\newcommand{\CS}{\mathcal{S}}
\newcommand{\FT}{\mathfrak{T}}

\newcommand{\al}{\alpha }

\newcommand{\ga}{\gamma }
\newcommand{\de}{\delta }
\newcommand{\Ga}{\Gamma }

\newcommand{\vth}{\vartheta }

\newcommand{\vph}{\varphi }
\newcommand{\la}{\lambda }
\newcommand{\La}{\Lambda }
\newcommand{\om}{\omega }
\newcommand{\Om}{\Omega }

\newcommand{\na}{\nabla }
\newcommand{\pa}{\partial }
\newcommand{\ot}{\otimes }

\newcommand{\bu}{\bullet}
\newcommand{\we}{\wedge}

\newcommand{\tpi}{2\pi \sqrt{-1}}
\newcommand{\pii}{\pi \sqrt{-1}}
\newcommand{\frs}{\mathfrak{s}}
\newcommand{\bff}{\mathbf{f}}
\newcommand{\bfv}{\mathbf{v}}
\newcommand{\bft}{\mathbf{t}}
\DeclareMathOperator{\diag}{diag}
\DeclareMathOperator{\reg}{reg}
\DeclareMathOperator{\Res}{Res}
\DeclareMathOperator{\ord}{ord}
\DeclareMathOperator{\id}{id}

\def\tp#1{\mathord{\mathopen{{\vphantom{#1}}^t}#1}} 

\makeatletter
\@addtoreset{equation}{section}

\makeatother

\title[Intersection numbers for the Riemann-Wirtinger integral]{
  Intersection numbers of twisted homology and cohomology groups 
  associated to the Riemann-Wirtinger integral
}
\author[Y. Goto]{Yoshiaki Goto}
\address[Goto]{
  General Education,
  Otaru University of Commerce,
  Otaru 047-8501, Japan
}
\email{goto@res.otaru-uc.ac.jp}

\date{\today}

\begin{document}
\begin{abstract}
  The Riemann-Wirtinger integral is an analogue of the hypergeometric integral, 
  which is defined as an integral 
  on a one-dimensional complex torus. 
  We study the intersection forms on the twisted homology and cohomology groups 
  associated with the Riemann-Wirtinger integral. 
  We derive explicit formulas of some intersection numbers, and 
  apply them to study the monodromy representation, connection problems, and contiguity relations. 
\end{abstract}

\maketitle

\section{Introduction}
The Gauss hypergeometric function ${}_2 F_1 (a,b,c;z)$ has an integral representation
\begin{align*}
  {}_2 F_1 (a,b,c;z)=\frac{\Ga(c)}{\Ga(a) \Ga(c-a)} \int_0^1 u^a (1-u)^{c-a} (1-zu)^{-b} \frac{du}{u(1-u)} .
\end{align*}
This integral can be interpreted as a pairing of a twisted homology class 
and a twisted cohomology class on $\P^1_u -\{ 0,1,1/z,\infty \}$. 
Aomoto (e.g., \cite{AK}) generalized this framework to 
twisted (co)homology theory on $\P^n$ minus some divisors, and it enables us to study 
many type of hypergeometric functions systematically. 
In addition, 
Kita-Yoshida \cite{KY} and Cho-Matsumoto \cite{CM} showed that 
the intersection forms on the twisted (co)homology groups are powerful tool to study hypergeometric functions, 
and gave technique to obtain explicit formulas of the intersection numbers 
for some (co)cycles. 
Thanks to the intersection theory on the twisted (co)homology groups, 
we can find various properties of hypergeometric functions, 
or interpret them in terms of the intersection forms. 

These theories give generalizations from $\P^1$ to $\P^n$. 
On the other hand, we can consider another type of generalization: from $\P^1$ to a Riemann surface 
of genus grater than zero. 
The Riemann-Wirtinger integral gives such a generalization on a one-dimensional complex torus. 

The Riemann-Wirtinger integral \cite{Mano}, \cite{Mano-Watanabe} is defined by 
\begin{align}
  \label{eq:RW-integral}
  \int_{\ga} e^{\tpi c_0 u}\vth_1(u-t_1)^{c_1} \cdots \vth_1(u-t_n)^{c_n} \frs(u-t_j;\la)du ,\\
  \nonumber
  \frs(u;\la) = \frac{\vth_1 (u-\la) \vth_1'(0)}{\vth_1(u)\vth_1(-\la)} ,
\end{align}
where $\ga$ is a twisted cycle, and $c_0\in \C$, $c_1,\dots ,c_n,\la \in \C-\Z$. 
For the theta function $\vth_1$, see \S \ref{subsec:theta-function}.
When $\la \in \Z$, we replace $\frs(u-t_j;\la)$ in (\ref{eq:RW-integral}) by $1$. 
The Wirtinger integral (e.g., \cite{Watanabe-elliptic-homology-cohomology}) is 
obtained as an example by setting $n=4$ and $\la =0$. 
The integrand of (\ref{eq:RW-integral}) can be regarded as a multi-valued function on 
a complex torus minus $n$ points for which we write $M$. 
It defines a local system $\CL_{\la}$ on $M$ and its dual $\CL_{\la}^{\vee}$. 
Thus, we can study the Riemann-Wirtinger integral in terms of 
the twisted cohomology group $H^1 (M;\CL_{\la})$ and the twisted homology group $H_1 (M;\CL_{\la}^{\vee})$. 
The structures of these groups are precisely studied in \cite{Mano-Watanabe}. 
Based on the results of \cite{Mano-Watanabe}, we study the intersection theory on 
these twisted homology and cohomology groups. 

In this paper, we compute the intersection numbers for various types of twisted (co)cycles. 
Moreover, we show that the intersection form on the twisted homology (resp. cohomology) group 
is useful to study the monodromy or connection problems (resp. contiguity relations) 
which are considered in \cite{Mano}. 
We can reduce these problems to studying certain linear operators on twisted (co)homology groups. 
We focus on twisted (co)cycles whose changes by a corresponding operator
are easily described. 
Then this operator can be expressed by using the intersection numbers with such (co)cycles. 
An advantage of our expressions is that they do not depend on the choice of a basis. 
By our expressions, we can recover the connection and contiguity matrices given in \cite{Mano}. 
Our approaches are analogies of \cite{M-FD} and \cite{GM-Pfaffian-contiguity}. 
The idea of using the intersection form to study the connection problems 
is also provided in \cite{Mimachi2011},  
which is slightly different from ours. 

In fact, some intersection numbers for twisted homology groups are computed in \cite{Ghazouani-Pirio}. 
Though we compute the intersection numbers of more cycles than \cite{Ghazouani-Pirio}, 
our results for the intersection numbers themselves are not essentially new. 
However, 
we show utility of the intersection form to study properties of the Riemann-Wirtinger integral, 
which is a different view point from \cite{Ghazouani-Pirio}. 
Further, 
we believe that it is the first time to study the intersection theory of twisted cohomology groups 
for our settings. 

This paper is arranged as follows. 
In Section \ref{sec:preliminaries}, 
we review basic properties of the theta function, 
and results of \cite{Mano-Watanabe} for the twisted homology and cohomology groups 
associated with the Riemann-Wirtinger integral. 
In Section \ref{sec:homology-intersection}, 
we give the intersection numbers of some twisted cycles, and apply them to study of 
the monodromy and connection problems. 
In Section \ref{sec:cohomology-intersection}, 
we give the intersection numbers of some twisted cocycles, and apply them to study of 
the contiguity relations. 
Precise computation of these intersection numbers are given in Section \ref{sec:computation}.

\section{Preliminaries}\label{sec:preliminaries}
In this section, we review basic facts about a theta function and twisted homology and cohomology groups,
which we will use throughout this paper. 

\subsection{Theta function}\label{subsec:theta-function}
We define a theta function 
\begin{align*}
  \vth_1 (u)=\vth_1 (u,\tau)
  =-\sum_{m\in \Z} \exp\left(\pi \sqrt{-1} \Big( m+\frac{1}{2} \Big)^2 \tau
  +\tpi \Big( m+\frac{1}{2} \Big)\Big( u+\frac{1}{2} \Big)\right), 
\end{align*}
where $z\in \C$ and $\tau \in \H$. 
In this paper, we fix $\tau \in \H$ and we frequently use the notation $\vth_1 (u)$. 
We note that $\vth_{1}(u,\tau)$ defined here is equal to $-\vth_{11}(u,\tau)$
which is introduced in \cite{Mumford}. 
It is well-known that $\vth_1 (u)$ is an odd function, has a simple zero at $u=0$, and 
has the quasi-periodicity 
\begin{align*}
  \vth_1(u+1)=-\vth_1(u) =e^{\pii}\vth_1(u),\quad 
  \vth_1(u+\tau)=-e^{-\pii (\tau+2u)}\vth_1(u)=e^{-\pii (\tau+2u+1)}\vth_1(u).
\end{align*}
We also introduce the following two functions: 
\begin{align*}
  \rho (u)=\frac{\vth_1'(u)}{\vth_1 (u)} ,\qquad 
  \frs (u;\la) =\frac{\vth_1 (u-\la) \vth_1'(0)}{\vth_1(u)\vth_1(-\la)} ,
\end{align*}
where $\vth_1'(u)=\frac{d}{du}\vth_1(u)$ and $\la \in \C-\Z$. 
Note that $\rho (u)$ is an odd function, and
$\frs (u;\la)$ has the quasi-periodicity
\begin{align*}
  \frs (u+1;\la) =\frs (u;\la) ,\quad 
  \frs (u+\tau;\la) =e^{\tpi \la} \frs (u;\la) .
\end{align*}
As a function in $\la$, the Laurent expansion of $\frs (u;\la)$ at $\la=0$ is given as
\begin{align}
  \label{eq:frs-laurent}
  \frs (u;\la) =-\frac{1}{\la} +\rho(u) +\cdots .
\end{align}
The following formulas are not difficult, and can be regarded as analogies of 
the pole-zero cancellation  
and the partial fraction decomposition in the rational functions field. 
\begin{Lem}[cf. {\cite[(38),(45)]{Mano}}]
  Suppose that $t_j,t_k,t_l \in \C$ are distinct points of $\C/(\Z +\Z \tau)$. 
  We have 
  \begin{align}
    \label{eq:theta-rel-1}
    &\frs(u-t_k;\la)\big( \rho(u-t_j)+\rho(t_j-t_k)-\rho(u-t_k-\la)-\rho(\la) \big)
      =\frs(u-t_j;\la)\frs(t_j-t_k;\la), \\
    \label{eq:theta-rel-2}
    &\frac{\vth_1(u-t_k)}{\vth_1(u-t_l)}\frs(u-t_j;\la-t_k+t_l)
      =\frac{\vth_1(t_j-t_k)}{\vth_1(t_j-t_l)}\frs(u-t_j;\la)
      +\frac{\vth_1(t_k-t_l)\vth_1(\la-t_k+t_j)}{\vth_1(t_j-t_l)\vth_1(\la-t_k+t_l)}\frs(u-t_l;\la).
  \end{align}
\end{Lem}

\subsection{Twisted homology and cohomology groups}\label{subsec:homology-cohomology}
For a fixed $\tau \in \H$, we set $\La_{\tau}=\Z +\Z \tau$ and $E=\C/\La_{\tau}$. 
For $\la \in \C$, we can define a one-dimensional representation $e_{\la} :\pi_1 (E) \simeq \La_{\tau} \to \C^{*}$ of 
the fundamental group $\pi_1 (E)$
by the correspondence $e_{\la}(1)=1$, $e_{\la}(\tau) =e^{\tpi \la}$. 
Let $R_{\la}$ be the local system on $E$ determined by this representation $e_{\la}$.  

Let $n\geq 2$ and let $t_1,\dots ,t_n$ be distinct points in $E$. 
We set $D=\{ t_1,\dots ,t_n \}$ and $M=E-D$. 
Let $\CO_E(*D)$ (resp. $\Om^1_E(*D)$) be the sheaf of functions (resp. $1$-forms) 
meromorphic on $E$ and holomorphic on $M$. 
We set $\CO_{\la}(*D)=\CO_E(*D) \ot_{\C} R_{\la}$ and $\Om^1_{\la}(*D)=\Om^1_E(*D) \ot_{\C} R_{\la}$. 
We define a multi-valued function $T(u)$ on $M$ by 
\begin{align*}
  T(u)= e^{\tpi c_0 u}\vth_1(u-t_1)^{c_1} \cdots \vth_1(u-t_n)^{c_n} , 
\end{align*}
where $c_0\in \C$, and $c_1,\dots ,c_n\in \C-\Z$ satisfy $c_1+\dots +c_n=0$. 
Let $\CL$ and $\CL^{\vee}$ be the local systems on $M$ defined by $T(u)^{-1}$ and $T(u)$, respectively: 
$\CL=\C T(u)^{-1}$ and $\CL^{\vee}=\C T(u)$. 
We set $\CL_{\la}=\CL \ot_{\C} R_{\la}$ and $\CL_{\la}^{\vee}=\CL^{\vee}\ot_{\C} R_{\la}^{\vee}$. 
Note that if $\la\in \Z$, we have $\CL_{\la}=\CL$ and $\CL_{\la}^{\vee}=\CL^{\vee}$. 
Let us consider the homology group $H_i(M;\CL_{\la}^{\vee})$ and cohomology group $H^i(M;\CL_{\la})$ which are 
called the twisted homology group and twisted cohomology group, respectively. 
Recall that the twisted homology group is defined as 
$H_i(M;\CL_{\la}^{\vee})=Z_i(M;\CL_{\la}^{\vee})/B_i(M;\CL_{\la}^{\vee})$, where $Z_i$ and $B_i$ are 
the kernel and image of the boundary operators for the twisted chains, respectively (e.g., \cite{AK}). 
We set $\om =d\log T(u)\in \Om^1_M(M)$. 

\begin{Fact}[\cite{Mano-Watanabe}]\label{fact:vanish-dim}
  If $i\neq 1$, then $H_i(M;\CL_{\la}^{\vee})=0$ and $H^i(M;\CL_{\la})=0$. We have 
  \begin{align*}
    \dim H_1(M;\CL_{\la}^{\vee}) =\dim H^1(M;\CL_{\la}) =n, \qquad
    H^1(M;\CL_{\la})\simeq \Om^1_{\la}(*D)(E)/\na(\CO_{\la}(*D)(E)) ,
  \end{align*}
  where $\na :\CO_{\la}(*D) \to \Om^1_{\la}(*D)$ is defined by $\na f =df +f\om$. 
\end{Fact}
Hereafter, we identify $H^1(M;\CL_{\la})$ with $\Om^1_{\la}(*D)(E)/\na(\CO_{\la}(*D)(E))$. 
We call an element in $Z_1(M;\CL_{\la}^{\vee})$ (resp. $\Om^1_{\la}(*D)(E)$) a twisted cycle (resp. cocycle). 
The natural pairing $H^1(M;\CL_{\la})\times H_1(M;\CL_{\la}^{\vee})
\ni ([\vph] ,[\ga]) \mapsto \int_{\ga} T(u)\vph \in \C$ is non-degenerate 
and gives the Riemann-Wirtinger integral.

The structure of $H^1(M;\CL_{\la})= \Om^1_{\la}(*D)(E)/\na(\CO_{\la}(*D)(E))$ is precisely studied in \cite{Mano-Watanabe}. 
As mentioned in \cite{Mano-Watanabe}, we may assume that $\la\in P=\{ a +b\tau \mid 0\leq a,b < 1 \}$ 
without loss of generality. We set 
\begin{align*}
  &\vph_0 (u;\la) =-\la \frs (u-t_1;\la) ,\qquad 
  \vph_1 (u;\la) =\frac{\pa\frs}{\pa u} (u-t_1;\la) ,\\ 
  &\vph_j (u;\la) =\frs (u-t_j;\la) -\frs (u-t_1;\la) \quad (j=2,\dots ,n).
\end{align*}
We can interpret $\vph_i (u;0)$ as $\lim_{\la\to 0} \vph_i (u;\la)$ which converges 
because of (\ref{eq:frs-laurent}). 
\begin{Fact}[\cite{Mano-Watanabe}]\label{fact:cohomology-basis}
  For any $\la\in P$, the $n+1$ classes $\{ [\vph_i (u;\la)du] \}_{i=0,\dots ,n}$ generate $H^1(M;\CL_{\la})$ 
  and satisfy a single relation\footnote{
As mentioned in \cite{Ghazouani-Pirio}, 
the relation written in \cite[p.\ 3876]{Mano-Watanabe} is not correct. 
The differences are ``$-\rho(t_j-t_1)$'' in the coefficient of $[\vph_0 (u;\la)du]$ and 
``$c_j$'' in the coefficient of $[\vph_j (u;\la)du]$. 
}
  \begin{align*}
    \Big( \tpi c_0 -c_1 \rho(\la) +\sum_{j=2}^n c_j \big(\frs(t_j-t_1;\la)-\rho(t_j-t_1)\big) \Big) [\vph_0 (u;\la)du]
    & \\
    +(c_1-1)\la [\vph_1 (u;\la)du]
    -\la \sum_{j=2}^n c_j \frs(t_j -t_1;\la) [\vph_j (u;\la)du] 
    &=0.
  \end{align*}
  In particular, we have the following properties. 
  \begin{enumerate}[{\rm (i)}]
  \item If $\la\in P-\{0\}$, then $\{ [\frs (u-t_j;\la)du] \}_{i=1,\dots ,n}$ form a basis of $H^1(M;\CL_{\la})$.
    Since $\frs (u-t_j;\la)du$ has a simple pole at $u=t_j$, 
    each element in $H^1(M;\CL_{\la})$ is represented by a $1$-form whose poles are of order $1$. 
  \item When $\la=0$, we have 
    \begin{align*}
      &\vph_0 (u;0)du =du ,\qquad 
      \vph_1 (u;0)du =\rho' (u-t_1)du ,\\ 
      &\vph_j (u;0)du =(\rho (u-t_j) -\rho (u-t_1))du \quad (j=2,\dots ,n),
    \end{align*}
    and the $1$-forms having only simple poles are not enough to generate $H^1(M;\CL_{\la})$. 
    The $1$-form $\rho' (u-t_1;\la)du$ which has a pole of order $2$ at $u=t_1$ is necessary to 
    give a basis. 
  \end{enumerate}
\end{Fact}

\begin{figure}
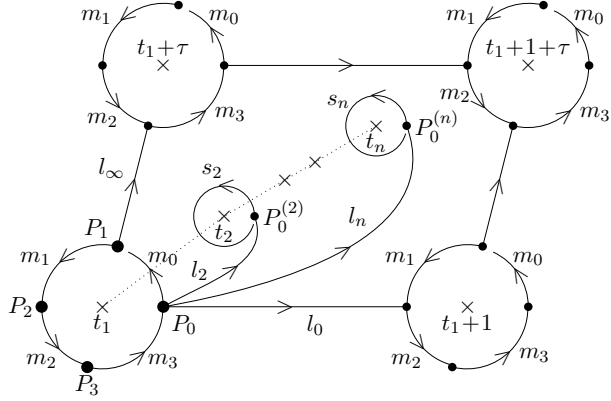

  \centering
  \scalebox{0.8}{\RWpictureA}
  \caption{The twisted cycles introduced in \cite{Mano-Watanabe}}
  \label{fig:cycle-1}
\end{figure}
Generators of $H_1(M;\CL_{\la}^{\vee})$ are also given in \cite{Mano-Watanabe}. 
We recall them. 
By definition, each twisted cycle is loaded with the information of the branch of $T(u)S(u)$ on 
its support, where $S(u)$ is a section of $R_{\la}$. 
Let $l_{\infty},l_0,l_2,\dots ,l_n$, $s_2,\dots ,s_n$, $m_0,m_1,m_2,m_3$ be $1$-chains drawn in Figure \ref{fig:cycle-1}. 
For a fixed branch of $T(u)S(u)$ at $P_0$, the branches on $l_0,l_2,\dots ,l_n, m_0$ and $m_3\cup m_2 \cup m_1$ are 
naturally defined. 
The branch on $s_j$ (resp. $l_{\infty}$) is defined so that the branch at the start point is same as 
that at the end point of $l_j$ (resp. start point of $m_1$).  
We note the images of some twisted $1$-chains under the boundary operator: 
\begin{align*}
  &\pa (m_1+m_2+m_3+m_0) =(e^{\tpi c_1}-1)P_1 ,&
  &\pa s_j =(e^{\tpi c_j}-1) P_0^{(j)} \quad (j=2,\dots ,n),& \\
  &\pa (m_2+m_3+l_0)=(e^{\tpi c_0}-1) P_2 ,&
  &\pa (-m_2-m_1+l_{\infty})=(e^{-\tpi c_{\infty}}-1) P_3 ,&
\end{align*}
where we put $c_{\infty}=-\la-c_0 \tau -c_1 t_1 -\dots -c_n t_n$. 
We define $n+1$ twisted cycles\footnote{
As mentioned in \cite{Ghazouani-Pirio}, 
we note that the coefficient of $(m_0+e^{\tpi c_1}m_1)$ in $\ga_0$ of \cite[p.\ 3877]{Mano-Watanabe}
should be $(1-e^{\tpi c_0})/(e^{\tpi c_1}-1)$.
}
as the regularization of paths joining points in $D$: 
\begin{align*}
  \ga_{1j} 
  &=\reg(t_1, t_j) =\frac{m_0+e^{\tpi c_1}(m_1+m_2+m_3)}{e^{\tpi c_1}-1} +l_j -\frac{s_j}{e^{\tpi c_j}-1} 
    \quad (j=2,\dots ,n), \\
  \ga_{10} 
  &=\reg(t_1,t_1+1)
    =\frac{m_0+e^{\tpi c_1}(m_1+m_2+m_3)}{e^{\tpi c_1}-1} +l_0 
    -e^{\tpi c_0}\frac{(m_2+m_3+m_0)+e^{\tpi c_1}m_1}{e^{\tpi c_1}-1} \\
  &=l_0+\frac{(1-e^{\tpi c_0})m_0+(1-e^{\tpi c_0})e^{\tpi c_1}m_1+(e^{\tpi c_1}-e^{\tpi c_0})(m_2+m_3)}{e^{\tpi c_1}-1} ,\\
  \ga_{1\infty} 
  &=\reg(t_1,t_1+\tau)
    =\frac{m_1+m_2+m_3+m_0}{e^{\tpi c_1}-1} +l_{\infty} 
    -e^{-\tpi c_{\infty}}\frac{(m_3+m_0)+e^{\tpi c_1}(m_1+m_2)}{e^{\tpi c_1}-1} \\
  &=l_{\infty}+\frac{(1-e^{-\tpi c_{\infty}})(m_3+m_0)+(1-e^{\tpi (c_1-c_{\infty})})(m_1+m_2)}{e^{\tpi c_1}-1} .
\end{align*}
\begin{Fact}[\cite{Mano-Watanabe}]
  The twisted homology group $H_1(M;\CL_{\la}^{\vee})$ is generated by $\{ [\ga_{1j}] \}_{j=2,\dots,n,0,\infty}$, 
  and these generators satisfy a single $\C$-linear relation
  \begin{align}
    \label{eq:homology-basis-relation}
    (e^{\tpi c_0}-1)[\ga_{1\infty}]+(1-e^{-\tpi c_{\infty}})[\ga_{10}]
    -\sum_{j=2}^n e^{-\tpi (c_1+\dots +c_j)}(1-e^{\tpi c_j})[\ga_{1j}] =0.
  \end{align}
\end{Fact}

\section{Intersection theory for twisted homology group}\label{sec:homology-intersection}
\subsection{Intersection form}\label{subsec:homology-intersection-form}
The homology intersection form $I_h$ is 
a non-degenerate bilinear form between $H_1(M;\CL_{\la}^{\vee})$ and $H_1(M;\CL_{\la})$: 
\begin{align*}
  I_h(\bu ,\bu) : H_1(M;\CL_{\la}^{\vee}) \times H_1(M;\CL_{\la}) \longrightarrow \C .
\end{align*}
For a twisted cycle $\ga \in Z_1(M;\CL_{\la}^{\vee})$, we can construct $\ga^{\vee} \in Z_1(M;\CL_{\la})$ by 
replacing $(\la,c_0,c_1,\dots ,c_n)$ with $(-\la,-c_0,-c_1,\dots ,-c_n)$
(and hence $c_{\infty}$ is also replaced by $-c_{\infty}$). 

According to \cite{KY}, the intersection number $I_h([\ga],[\de])$ for
$[\ga]\in H_1(M;\CL_{\la}^{\vee})$ and $[\de]\in H_1(M;\CL_{\la})$ is evaluated as follows. 
We may assume that twisted cycles $\ga$ and $\de$ are expressed as 
\begin{align*}
  \ga = \sum_{i} a_i \cdot \Delta_i \ot (TS)_{\Delta_i} ,\quad 
  \de = \sum_{j} a'_j \cdot \Delta'_j \ot (TS)_{\Delta'_j}^{-1}
  \qquad (a_i,a'_j\in \C),
\end{align*}
where $\Delta_i$ and $\Delta'_j$ are simply connected and  
the intersection $\Delta_i \cap \Delta'_j$ is at most one point 
at which they intersect transversally, and 
$(TS)_{\Delta_i}$ denotes the branch of $T(u)S(u)$ on $\Delta_i$. 
Then the intersection number is evaluated as 
\begin{align*}
  I_h([\ga],[\de])
  =\sum_{\Delta_i \cap \Delta'_j =\{ u_{ij} \} }
  a_i a'_j \cdot
  I_{\rm loc}(\Delta_i , \Delta'_j) \cdot 
  (TS)_{\Delta_i}(u_{ij}) \cdot (TS)_{\Delta'_j}(u_{ij})^{-1} ,
\end{align*}
where $I_{\rm loc}$ is the local intersection multiplicity.

Note that we often consider $c_j$ as an indeterminate and we can regard an intersection number 
as an element in the field $\C(e^{\tpi c_{*}}):=\C(e^{\tpi c_1},\dots ,e^{\tpi c_n},e^{\tpi c_0},e^{\tpi c_{\infty}})$. 
In such a situation, 
for $[\ga],[\de]\in H_1(M;\CL_{\la}^{\vee})$, we have 
$I_h([\ga],[\de^{\vee}])=-I_h([\de],[\ga^{\vee}])^{\vee}$, where 
the last $\vee$ stands for the involution on $\C(e^{\tpi c_{*}})$
given by $(c_1,\dots,c_n,c_0,c_{\infty}) \mapsto (-c_1,\dots,-c_n,-c_0,-c_{\infty})$. 

\subsection{Intersection numbers}
In this section, we give formulas of the intersection numbers for 
the twisted cycles introduced in \S \ref{subsec:homology-cohomology}. 
Though they are computed in \cite{Ghazouani-Pirio} essentially, we rewrite them in our notations. 
We also define other twisted cycles which will be used in \S \ref{subsec:monodromy}, 
and give their intersection numbers. 
Precise computations will be given in \S \ref{subsec:eval-homology-intersection}. 

\begin{Fact}[{\cite[Proposition 3.4.1]{Ghazouani-Pirio}}]\label{fact:homology-intersection-basis}
  For $j,k\in \{2,\dots,n\}$, we have 
  \begin{align*}
    &I_h ([\ga_{1j}] ,[\ga_{1k}^{\vee}])=
    \begin{cases}
      \frac{e^{\tpi c_1}}{1-e^{\tpi c_1}} & (j<k) \\
      \frac{1}{1-e^{\tpi c_1}} & (j>k),
    \end{cases} \qquad 
    I_h ([\ga_{1j}] ,[\ga_{1j}^{\vee}])
      =\frac{1-e^{\tpi (c_1 +c_j)}}{(1-e^{\tpi c_1})(1-e^{\tpi c_j})} ,\\
    &I_h ([\ga_{1j}] ,[\ga_{10}^{\vee}]) = \frac{e^{\tpi c_1}(1-e^{-\tpi c_0})}{1-e^{\tpi c_1}} ,\qquad 
      I_h ([\ga_{10}] ,[\ga_{1j}^{\vee}]) =\frac{1-e^{\tpi c_0}}{1-e^{\tpi c_1}}, \\
    &I_h ([\ga_{1j}] ,[\ga_{1\infty}^{\vee}]) = \frac{e^{\tpi c_1}(1-e^{\tpi c_{\infty}})}{1-e^{\tpi c_1}} ,\qquad 
      I_h ([\ga_{1\infty}] ,[\ga_{1j}^{\vee}]) =\frac{1-e^{-\tpi c_{\infty}}}{1-e^{\tpi c_1}}, \\
    &I_h ([\ga_{10}] ,[\ga_{10}^{\vee}])
      =-\frac{(e^{\tpi c_0}-1)(e^{\tpi c_0}-e^{\tpi c_1})}{e^{\tpi c_0}(1-e^{\tpi c_1})} , \\ 
    &I_h ([\ga_{1\infty}] ,[\ga_{1\infty}^{\vee}])
      =-\frac{(e^{\tpi c_{\infty}}-1)(e^{\tpi c_{\infty}}-e^{\tpi c_1})}{e^{\tpi c_{\infty}}(1-e^{\tpi c_1})} ,\\
    &I_h ([\ga_{10}] ,[\ga_{1\infty}^{\vee}]) 
      =\frac{e^{\tpi c_1}-e^{\tpi (c_0+c_1)}-e^{\tpi (c_1 +c_{\infty})}+e^{\tpi (c_0+c_{\infty})}}{1-e^{\tpi c_1}}, \\
    &I_h ([\ga_{1\infty}] ,[\ga_{10}^{\vee}]) 
      =\frac{1-e^{-\tpi c_{\infty}}-e^{-\tpi c_0}+e^{\tpi (-c_0+c_1-c_{\infty})}}{1-e^{\tpi c_1}}. 
  \end{align*}
  In particular, the determinant of 
  the intersection matrix $H_{11}=\big( I_h ([\ga_{1j}] ,[\ga_{1k}^{\vee}]) \big)_{j,k= 2,\dots,n-1,0,\infty }$ is equal to 
  \begin{align*}
    \frac{1-e^{\tpi (c_1+\cdots +c_{n-1})}}{(1-e^{\tpi c_1})\cdots(1-e^{\tpi c_{n-1}})}
    =\frac{1-e^{-\tpi c_n}}{(1-e^{\tpi c_1})\cdots(1-e^{\tpi c_{n-1}})}.
  \end{align*}
\end{Fact}
This determinant formula implies that $\{ [\ga_{1j}] \}_{j= 2,\dots,n-1,0,\infty}$ give a basis 
of $H_1(M;\CL_{\la}^{\vee})$ under the condition $c_1,\dots ,c_n\not\in \Z$. 
This is another approach to \cite[Theorem 3.1]{Mano-Watanabe}. 

\begin{Rem}
  Let $H_{11,0\infty}=\big( I_h ([\ga_{1j}] ,[\ga_{1k}^{\vee}]) \big)_{j,k=0,\infty }$ be a 
  $2\times 2$ submatrix of $H_{11}$. 
  By straightforward calculation, we have $\det (H_{11,0\infty})=1$. 
  Furthermore, if we put $c_0=c_{\infty}=0$, then we have 
  $H_{11,0\infty}=\left(
    \begin{smallmatrix}
      0 & 1 \\ -1 & 0
    \end{smallmatrix} \right)$. 
  Indeed, $\{\ga_{10},\ga_{1\infty}\}$ can be naturally identified with 
  a symplectic basis of the usual homology group $H_1 (E;\Z)$, 
  since the coefficients of $m_0,m_1$ (resp. $m_0,m_3$) in the definition of $\ga_{10}$ (resp. $\ga_{1\infty}$)
  become zero. 
\end{Rem}

For $2\leq j <k \leq n$, we set 
$\ga_{jk}=\reg (t_j,t_k) =\ga_{1k}-\ga_{1j}$. 
As a corollary of Fact \ref{fact:homology-intersection-basis}, we easily obtain the following. 
\begin{Cor}\label{cor:homology-intersection-interval}
  For $j,j',k,k'\in \{1,\dots,n\}$ satisfying $j<k$, $j'<k'$, the intersection number 
  $I_h ([\ga_{jk}] ,[\ga_{j'k'}^{\vee}])$ is given as follows. 
  \begin{itemize}
  \item If $(j,k)=(j',k')$, then $I_h ([\ga_{jk}] ,[\ga_{jk}^{\vee}])
    =\frac{e^{\tpi (c_j+c_k)}}{(1-e^{\tpi c_j})(1-e^{\tpi c_k})}$. 
  \item If $j=j'$, then 
    $I_h ([\ga_{jk}] ,[\ga_{jk'}^{\vee}])=
    \begin{cases}
      \frac{e^{\tpi c_j}}{1-e^{\tpi c_j}} & (k<k') \\
      \frac{1}{1-e^{\tpi c_j}} & (k>k') .
    \end{cases}$
  \item If $k=k'$, then 
    $I_h ([\ga_{jk}] ,[\ga_{j'k}^{\vee}])=
    \begin{cases}
      \frac{e^{\tpi c_k}}{1-e^{\tpi c_k}} & (j<j') \\
      \frac{1}{1-e^{\tpi c_k}} & (j>j') .
    \end{cases}$
  \item If $k=j'$, then $I_h ([\ga_{jk}] ,[\ga_{kk'}^{\vee}])=\frac{-1}{1-e^{\tpi c_k}}$. 
  \item If $j=k'$, then $I_h ([\ga_{jk}] ,[\ga_{j'j}^{\vee}])=\frac{-e^{\tpi c_j}}{1-e^{\tpi c_j}}$. 
  \item If $j<j'<k<k'$ (resp. $j'<j<k'<k$), then $I_h ([\ga_{jk}] ,[\ga_{j'k'}^{\vee}])=-1$
    (resp. $1$). 
  \item Otherwise, $I_h ([\ga_{jk}] ,[\ga_{j'k'}^{\vee}])=0$. 
  \end{itemize}
  If $2\leq j<k \leq n$, then we have 
  \begin{align*}
    I_h ([\ga_{jk}] ,[\ga_{10}^{\vee}])
    =I_h ([\ga_{jk}] ,[\ga_{1\infty}^{\vee}])
    =I_h ([\ga_{10}] ,[\ga_{jk}^{\vee}])
    =I_h ([\ga_{1\infty}] ,[\ga_{jk}^{\vee}])
    =0.
  \end{align*}
\end{Cor}
Let $H_1$ be the intersection matrix for $\ga_{23}, \ga_{34},\dots, \ga_{n-1,n},\ga_{10},\ga_{1\infty}$. 
Then $H_1$ is of the form 
$\left(
\begin{smallmatrix}
  H_1' & O \\ O & H_{11,0\infty}
\end{smallmatrix} \right)$ with a tridiagonal matrix $H_1'$. 
As evaluated in \cite[Remark 2.2]{KY}, we have 
\begin{align*}
  \det H_1' 
  = \frac{1-e^{\tpi (c_2+\cdots +c_n)}}{(1-e^{\tpi c_2})\cdots(1-e^{\tpi c_n})}
  = \frac{1-e^{-\tpi c_1}}{(1-e^{\tpi c_2})\cdots(1-e^{\tpi c_n})} ,
\end{align*}
and hence, $[\ga_{23}], [\ga_{34}],\dots, [\ga_{n-1,n}],[\ga_{10}],[\ga_{1\infty}]$ also give a basis of 
$H_1(M;\CL_{\la}^{\vee})$ under the condition $c_1,\dots ,c_n\not\in \Z$. 

\begin{Rem}
  \begin{enumerate}[(1)]
  \item Note that $H_1'$ coincides with the intersection matrix for the basis 
    $\{ \reg (x_j,x_{j+1}) \}_{j=2,\dots ,n-1}$ of the twisted homology group 
    associated with the multi-valued function $(z-x_2)^{c_2}\cdots (z-x_n)^{c_n}$ on 
    $\P^1_z-\{x_1=\infty,x_2,\dots ,x_n \}$ (see, e.g., \cite[Theorem 2.1]{KY}). 
  \item If $c_1,\dots ,c_n ,c_0,c_{\infty}\in \R$, then $\sqrt{-1}H_1$ is a monodromy invariant Hermitian matrix 
    in the sense of \cite[\S 2.5]{G-Matsubara}. 
    It is easy to see that the signature of $\sqrt{-1}H_{11,0\infty}$ is $(1,1)$. 
    Thus, $\sqrt{-1}H_1$ is indefinite for any real parameters. 
  \end{enumerate}  
\end{Rem}

We next define twisted cycles $\ga_{j0}$ and $\ga_{j\infty}$ ($2\leq j \leq n$) in a similar manner to 
that for $\ga_{10}$ and $\ga_{1\infty}$, respectively. 
Let $l_{\infty}^{(j)},l_0^{(j)}$, $m_0^{(j)},m_1^{(j)},m_2^{(j)},m_3^{(j)}$ be $1$-chains drawn in Figure \ref{fig:cycle-j}.
Note that we have $s_j=m_0^{(j)}+e^{\tpi c_j}(m_1^{(j)}+m_2^{(j)}+m_3^{(j)})$. 
We set 
\begin{align*}
  \ga_{j0} 
  &=\reg(t_j,t_j+1) \\
  &=\frac{m_0^{(j)}+e^{\tpi c_j}(m_1^{(j)}+m_2^{(j)}+m_3^{(j)})}{e^{\tpi c_j}-1} +l_0^{(j)} \\
  &\qquad 
    -e^{\tpi c_0}e^{\tpi (-c_1-\dots -c_{j-1})} \cdot \frac{(m_2^{(j)}+m_3^{(j)}+m_0^{(j)})
    +e^{\tpi c_j}m_1^{(j)}}{e^{\tpi c_j}-1} \\
  &=l_0^{(j)}+\frac{1}{e^{\tpi c_j} -1}\Big(
  ( 1-e^{\tpi (c_0-c_1-\cdots -c_{j-1})})m_0^{(j)} 
  +(e^{\tpi c_j}-e^{\tpi (c_0-c_1-\cdots -c_{j-1}+c_j)})m_1^{(j)} \\
  &\qquad \qquad 
    +(e^{\tpi c_j}-e^{\tpi (c_0-c_1-\cdots -c_{j-1})})(m_2^{(j)}+m_3^{(j)})
  \Big) ,\\
  \ga_{j\infty} 
  &=\reg(t_j,t_j+\tau) \\
  &=\frac{m_1^{(j)}+m_2^{(j)}+m_3^{(j)}+m_0^{(j)}}{e^{\tpi c_j}-1} +l_{\infty}^{(j)}\\
  &\qquad  
    -e^{-\tpi c_{\infty}}e^{\tpi (c_1+\dots +c_{j-1})} \cdot 
    \frac{(m_3^{(j)}+m_0^{(j)})+e^{\tpi c_j}(m_1^{(j)}+m_2^{(j)})}{e^{\tpi c_j}-1} \\
  &=l_{\infty}^{(j)}+\frac{1}{e^{\tpi c_j} -1}\Big(
  ( 1-e^{\tpi (-c_{\infty}+c_1+\cdots +c_{j-1})})(m_0^{(j)}+m_3^{(j)}) \\
  &\qquad \qquad 
    +(1-e^{\tpi (-c_{\infty}+c_1+\cdots +c_{j-1}+c_j)})(m_1^{(j)}+m_2^{(j)}) 
    \Big) .
\end{align*}
\begin{figure}
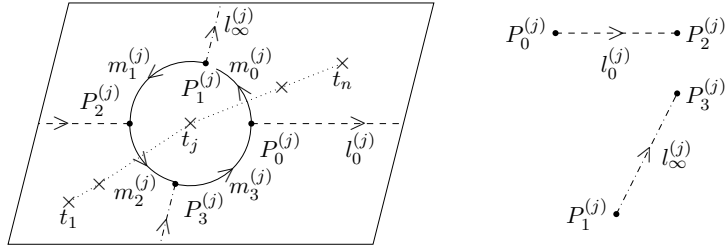

  \centering
  \scalebox{0.8}{\RWpictureB}
  \caption{The twisted cycles $\ga_{j0}$ and $\ga_{j\infty}$}
  \label{fig:cycle-j}
\end{figure}
\begin{Prop}\label{prop:homology-int-diag-0infty}
  We set $H_{00}=(I_h ([\ga_{j0}] ,[\ga_{k0}^{\vee}]))_{j,k=1,\dots ,n}$ and 
  $H_{\infty \infty}=(I_h ([\ga_{j\infty}] ,[\ga_{k\infty}^{\vee}]))_{j,k=1,\dots ,n}$. 
  These are diagonal matrices whose diagonal entries are 
  \begin{align*}
    I_h ([\ga_{j0}] ,[\ga_{j0}^{\vee}])
    &=-\frac{(e^{\tpi c_0}-e^{\tpi (c_1+\cdots +c_{j-1})})(e^{\tpi c_0}-e^{\tpi (c_1+\cdots +c_j)})}
      {e^{\tpi (c_0+c_1+\cdots +c_{j-1})}(1-e^{\tpi c_j})}, \\
    I_h ([\ga_{j\infty}] ,[\ga_{j\infty}^{\vee}])
    &=-\frac{(e^{\tpi c_{\infty}}-e^{\tpi (c_1+\cdots +c_{j-1})})(e^{\tpi c_{\infty}}
      -e^{\tpi (c_1+\cdots +c_{j})})}{e^{\tpi (c_{\infty}+c_1+\cdots +c_{j-1})}(1-e^{\tpi c_j})}.
  \end{align*}
\end{Prop}
This implies that $\{ [\ga_{j0}] \}_{j=1,\dots,n}$ (resp. $\{ [\ga_{j\infty}] \}_{j=1,\dots ,n}$) form a basis of 
$H_1(M;\CL_{\la}^{\vee})$ when we assume the conditions not only $c_1,\dots ,c_n\not\in \Z$ but also 
$c_0-c_1-\cdots -c_{j} \not\in \Z$
(resp. $c_{\infty}-c_1-\cdots -c_{j} \not\in \Z$) for $1\leq j \leq n$.

\begin{Rem}
  These additional conditions are indispensable. 
  For example, we can easily verify the relation $[\ga_{20}]-[\ga_{10}]=(e^{\tpi (c_0-c_1)}-1)[\ga_{12}]$
  which implies $[\ga_{20}]=[\ga_{10}]$ if $c_0-c_1\in \Z$. 
\end{Rem}

Let $l_1^{(n)},\dots ,l_{n-1}^{(n)}$ be $1$-chains drawn in Figure \ref{fig:cycle-n}. 
For $1\leq j \leq n-1$, we set 
\begin{align*}
  \ga_{nj}=\frac{s_n}{e^{\tpi c_n}-1} +l_j^{(n)}
  -\frac{e^{\tpi (c_{j+1}+\dots +c_n)}s_j}{e^{\tpi c_j}-1} .
\end{align*}
Note that $[\ga_{nj}]\neq -[\ga_{jn}]$ in general. 
\begin{figure}
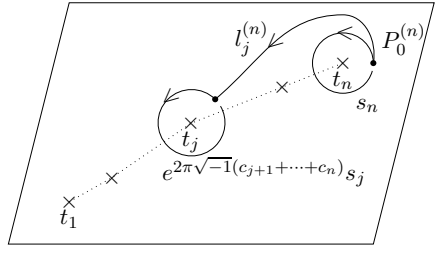

  \centering
  \scalebox{0.8}{\RWpictureC}
  \caption{The twisted cycle $\ga_{nj}$.}
  \label{fig:cycle-n}
\end{figure}
\begin{Prop}\label{prop:homology-int-diag-1n}
  The matrix $H_{1n}=(I_h ([\ga_{1j}] ,[\ga_{nk}^{\vee}]))_{\substack{j=2,\dots ,n-1 ,0,\infty \\ k=2,\dots ,n-1 ,\infty,0}}$
  is a diagonal one whose diagonal entries are given by 
  \begin{align*}
    I_h ([\ga_{1j}] ,[\ga_{nj}^{\vee}])
    &=\frac{e^{\tpi (c_1+\cdots +c_j)}}{1-e^{\tpi c_j}} \quad (j=2,\dots ,n-1) ,\\
    I_h ([\ga_{10}] ,[\ga_{n\infty}^{\vee}])
    &=e^{\tpi (c_n+c_{\infty})} ,\qquad 
      I_h ([\ga_{1\infty}] ,[\ga_{n0}^{\vee}])
      =-e^{-\tpi c_0}. 
  \end{align*}
\end{Prop}
Thus, we can express an arbitrary twisted cycle $\ga$ as a linear combination of 
the basis $\{ [\ga_{1j}] \}_{j= 2,\dots,n-1,0,\infty}$ by using intersection numbers: 
\begin{align*}
  [\ga]
  &=\sum_{k=2}^{n-1} \frac{I_h([\ga] ,[\ga_{nk}^{\vee}])}{I_h([\ga_{1k}] ,[\ga_{nk}^{\vee}])} [\ga_{1k}]
  +\frac{I_h([\ga] ,[\ga_{n\infty}^{\vee}])}{I_h([\ga_{10}] ,[\ga_{n\infty}^{\vee}])} [\ga_{10}]
  +\frac{I_h([\ga] ,[\ga_{n0}^{\vee}])}{I_h([\ga_{1\infty}] ,[\ga_{n0}^{\vee}])} [\ga_{1\infty}] \\
  &=\sum_{k=2}^{n-1} \frac{1-e^{\tpi c_j}}{e^{\tpi (c_1+\cdots +c_j)}}I_h([\ga] ,[\ga_{nk}^{\vee}])[\ga_{1k}]
  +e^{\tpi (-c_n-c_{\infty})}I_h([\ga] ,[\ga_{n\infty}^{\vee}]) [\ga_{10}]
  -e^{\tpi c_0} I_h([\ga] ,[\ga_{n0}^{\vee}]) [\ga_{1\infty}].
\end{align*}
For example, if we set $\ga =\ga_{1n}$, then this expression yields the relation 
(\ref{eq:homology-basis-relation}).

\subsection{Monodromy and connection problem}\label{subsec:monodromy}
In this section, we consider monodromy and connection problems by moving $t_p$'s. 
Though some of the representation matrices for the basis $\{ [\ga_{1j}] \}_{j= 2,\dots,n-1,0,\infty}$ are 
given in \cite[\S 6]{Mano}, we reconsider them by using the intersection form. 
To describe the move of $t_p$'s, we set 
\begin{align*}
  \FT=\{ (t_1,\dots ,t_n)\in \C^n \mid t_j-t_k \not\in \La_{\tau}  \ (j\neq k) \} .
\end{align*}
We choose a base point $\bft^{\circ}=(t_1^{\circ},\dots ,t_n^{\circ})\in \FT$ such that 
each $t_j^{\circ}$ belongs to $P=\{ a +b\tau \mid 0\leq a,b < 1 \}$. 

\subsubsection{Monodromy}
For $1\leq p<q \leq n$, let $\ell_{pq}$ be a loop in 
\begin{align*}
  \{(t_1^{\circ},\dots ,t_{p-1}^{\circ},s,t_{p+1}^{\circ},\dots ,t_n^{\circ})\in \FT \mid 
  s\in P ,\ s\neq t_j^{\circ} \ (j\neq p) \} \subset \FT
\end{align*}
with terminal $\bft^{\circ}$ starting from $s=t_p^{\circ}$ approaching to $t_q^{\circ}$ 
via the right side of the branch cut, turning around this point counterclockwisely, and 
tracing back to $t_p^{\circ}$ (see Figure \ref{fig:loop-pq}). 
The loop $\ell_{pq}$ naturally induces a linear automorphism 
$\ell_{pq*}:H_1(M;\CL_{\la}^{\vee})\to H_1(M;\CL_{\la}^{\vee})$
which is called the circuit transformation. 
We give an expression of $\ell_{pq*}$ by using the intersection form.  
\begin{figure}
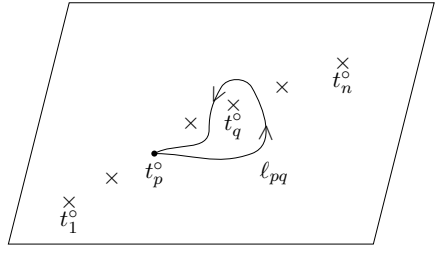

  \centering
  \scalebox{0.8}{\RWpictureD}
  \caption{The loop $\ell_{pq}$ (in the $s$-coordinate).}
  \label{fig:loop-pq}
\end{figure}

\begin{Th}\label{th:monodromy}
  For an arbitrary $[\ga] \in H_1(M;\CL_{\la}^{\vee})$, we have
  \begin{align*}
    \ell_{pq*} ([\ga]) =[\ga] - (1-e^{\tpi c_p})(1-e^{\tpi c_q}) I_h ([\ga] ,[\ga_{pq}^{\vee}]) [\ga_{pq}]. 
  \end{align*}
\end{Th}
The proof of this theorem is quite similar to that of \cite[Theorem 5.1]{M-FD}, since 
there are few differences between $\P^1$ and $E$ 
as far as we consider the loop $\ell_{pq}$.
\begin{Lem}\label{lem:monodromy-eigen}
  \begin{enumerate}
  \item $[\ga_{pq}]$ is an eigenvector of $\ell_{pq*}$ with eigenvalue $e^{\tpi (c_p +c_q)}$. 
  \item We set $\ga_{pq}^{\perp}=\{ [\ga]\in H_1(M;\CL_{\la}^{\vee}) \mid I_h([\ga],[\ga_{pq}^{\vee}])=0 \}$. 
    Then $\dim (\ga_{pq}^{\perp})=n-1$. 
  \item Further, $\ga_{pq}^{\perp}$ coincides with the eigenspace of $\ell_{pq*}$ with 
    eigenvalue one. 
  \end{enumerate}
\end{Lem}
\begin{proof}
  \begin{enumerate}[(1)]
  \item We can prove this claim in a similar way to the proof of \cite[Lemma 5.1]{M-FD}. 
  \item 
    Since the non-degenerate property of the intersection form implies $\dim (\ga_{pq}^{\perp})\leq n-1$, 
    it suffices to find $n-1$ linearly independent twisted cycles belonging to $\ga_{pq}^{\perp}$.
    First, we assume $n\geq 3$. 
    If $q\neq n$, 
    then $n-1$ cycles 
    $\{[\ga_{nj}]\}_{j=1,\dots,\check{p},\dots,\check{q},\dots ,n-1,0,\infty}$
    belong to $\ga_{pq}^{\perp}$. 
    If $q=n$ and $p\neq 1$, then 
    $\{ [\ga_{n1}-\ga_{nj}] \}_{j=2,\dots,\check{p},\dots,n-1}\cup \{ [\ga_{10}],[\ga_{1\infty}]\}$
    belong to $\ga_{pq}^{\perp}$. 
    Thus, in these cases, we can easily check $\dim (\ga_{pq}^{\perp})=n-1$. 
    We consider the case when $(p,q)=(1,n)$. 
    We set $\ga_{20}^{(1n)}=e^{\tpi c_1}\ga_{20} + (e^{\tpi c_1}-1)\ga_{12}$ and 
    $\ga_{2\infty}^{(1n)}=\ga_{2\infty} - e^{-\tpi c_{\infty}}(e^{\tpi c_1}-1)\ga_{12}$ 
    which are homologous to the twisted cycles drawn in Figure \ref{fig:cycle-avoid}. 
    \begin{figure}
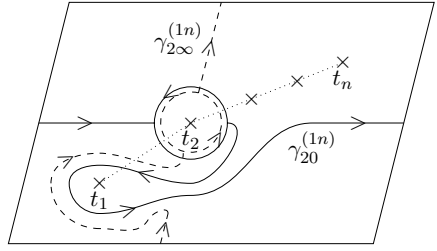

      \centering
      \scalebox{0.8}{\RWpictureE}
      \caption{The cycle drawn by solid (resp. dashed) line is homologous to $\ga_{20}^{(1n)}$ (resp. $\ga_{2\infty}^{(1n)}$)}
      \label{fig:cycle-avoid}
    \end{figure}
    It is not difficult to show that 
    $\{ [\ga_{2j}] \}_{j=3,\dots ,n-1}\cup \{ [\ga_{20}^{(1n)}],[\ga_{2\infty}^{(1n)}]\}$ are linearly independent 
    and belong to $\ga_{1n}^{\perp}$. 
    Thus, we obtain $\dim (\ga_{1n}^{\perp})=n-1$. 
    \\
    Next, we consider the case when $n=2$ and $(p,q)=(1,2)$. 
    In this case, we have $e^{\tpi (c_1+c_2)}=1$, and hence $\ga_{12}$ itself belongs to $\ga_{12}^{\perp}$
    by Fact \ref{fact:homology-intersection-basis}. 
    Thus, $\ga_{12}^{\perp}$ has a positive dimension, 
    which implies $\dim (\ga_{12}^{\perp})=1$. 
  \item 
    If $n\geq 3$, it is clear that 
    each of the above bases of $\ga_{pq}^{\perp}$ is not changed under $\ell_{pq*}$. 
    Thus, $\ga_{pq}^{\perp}$ is contained in the eigenspace of $\ell_{pq*}$ with eigenvalue one. 
    Since $\ell_{pq*}$ is not the identity map, we obtain the claim. 
    \\
    When $n=2$ and $(p,q)=(1,2)$, we have $\ga_{12}^{\perp} =\C \cdot [\ga_{12}]$ by the proof of (2). 
    Thus, $e^{\tpi (c_1+c_2)}=1$ and (1) shows the claim (3). 
    \qedhere
  \end{enumerate}
\end{proof}
\begin{proof}[Proof of Theorem \ref{th:monodromy}]
  By using Lemma \ref{lem:monodromy-eigen} and Corollary \ref{cor:homology-intersection-interval}, 
  we can show the theorem in the same way as the proof of \cite[Theorem 5.1]{M-FD}. 
\end{proof}
Applying Theorem \ref{th:monodromy}, we can obtain the representation matrix $M_{pq}$ of $\ell_{pq*}$
with respect to the basis $\{ [\ga_{1j}] \}_{j= 2,\dots,n-1,0,\infty}$. 
\begin{Cor}
  We set 
  \begin{align*}
    \bfv_1&=\tp{(0,\dots,0)}, \quad 
            \bfv_j=\tp{(}0,\dots,0,\overset{(j-1)\text{-th}}{1},0,\dots,0) \ (j\neq 1,n) ,\\
    \bfv_n&=\frac{1}{1-e^{\tpi c_n}}\tp{\big(}
            -e^{-\tpi (c_1+c_2)}(1-e^{\tpi c_2}) ,\dots ,-e^{-\tpi (c_1+\cdots+c_{n-1})}(1-e^{\tpi c_{n-1}}),\\
          &\qquad \qquad \qquad \qquad \qquad 
            1-e^{-\tpi c_{\infty}}, e^{\tpi c_0}-1
            \big), 
  \end{align*}
  and $\bfv_{jk}=\bfv_k -\bfv_j$. 
  Then $M_{pq}:=\id_n -(1-e^{\tpi c_p})(1-e^{\tpi c_q}) \bfv_{pq} \tp{\bfv_{pq}^{\vee}}\tp{H_{11}}$ satisfies
  \begin{align}
    \label{eq:rep-mat-ellpq}
    &\ell_{pq*}([\ga_{12}],\dots ,[\ga_{1,n-1}],[\ga_{10}],[\ga_{1\infty}])
      =([\ga_{12}],\dots ,[\ga_{1,n-1}],[\ga_{10}],[\ga_{1\infty}])M_{pq}, \\
    \label{eq:relation-mat-ellpq}
    &H_{11}=\tp{M_{pq}} H_{11} M_{pq}^{\vee}, 
  \end{align}
  where $\id_n$ is the unit matrix of size $n$ and $H_{11}$ is the intersection matrix defined 
  in Fact \ref{fact:homology-intersection-basis}. 
\end{Cor}
\begin{proof}
  By definition and (\ref{eq:homology-basis-relation}), we have
  $[\ga_{pq}]=([\ga_{12}],\dots ,[\ga_{1\infty}])\bfv_{pq}$. 
  If $[\ga]$ is expressed as $[\ga]=([\ga_{12}],\dots ,[\ga_{1\infty}])\bfv$ for some $\bfv$, then 
  the intersection number $I_h([\ga],[\ga_{pq}^{\vee}])$ is given by 
  $\tp{\bfv}H_{11}\bfv_{pq}^{\vee}=\tp{\bfv_{pq}^{\vee}}\tp{H_{11}}\bfv$. 
  This implies that $I_h ([\ga] ,[\ga_{pq}^{\vee}]) [\ga_{pq}]$ can be expressed as 
  $([\ga_{12}],\dots ,[\ga_{1\infty}])\bfv_{pq} \tp{\bfv_{pq}^{\vee}}\tp{H_{11}}\bfv$, 
  and hence we obtain (\ref{eq:rep-mat-ellpq}) by Theorem \ref{th:monodromy}. 
  The equality (\ref{eq:relation-mat-ellpq}) follows from the monodromy invariant property 
  $I_h([\ga],[\de])=I_h(\ell_{pq*}([\ga]),\ell_{pq*}^{\vee}([\de]))$, 
  where $\ell_{pq*}^{\vee}:H_1(M;\CL_{\la})\to H_1(M;\CL_{\la})$ is the circuit transformation. 
\end{proof}

\subsubsection{Translating $t_p$ to $t_p +1$}\label{subsubsec:connection+1}
For $p=1,\dots ,n$, let $\ell_{p0}$ be a path  
\begin{align*}
  [0,1]\ni s \mapsto (t_1^{\circ},\dots ,t_{p-1}^{\circ},t_{p}^{\circ}+s,t_{p+1}^{\circ},\dots ,t_n^{\circ})\in \FT
\end{align*}
from $\bft^{\circ}$ to 
$\bft^{\circ}_{p0}=(t_1^{\circ},\dots ,t_{p-1}^{\circ},t_{p}^{\circ}+1,t_{p+1}^{\circ},\dots ,t_n^{\circ})$
(see Figure \ref{fig:path-p0}). 
\begin{figure}
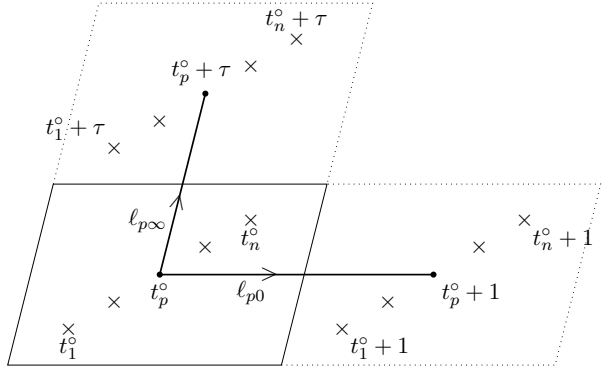

  \centering
  \scalebox{0.8}{\RWpictureF}
  \caption{The paths $\ell_{p0}$ and $\ell_{p\infty}$.}
  \label{fig:path-p0}
\end{figure}
Following \cite{Mano}, we consider $c_{\infty}\in \C$ as a constant, 
and we move $s$ from $0$ to $1$ keeping the constraint 
$\la +c_0 \tau +c_1 t_1 +\cdots +c_n t_n +c_{\infty}=0$. 
When $s$ is $1$, the local system $\CL_{\la}^{\vee}$ changes into 
$\CL_{\la}^{(p0)\vee} :=\C T(u)|_{t_p \to t_p +1} \ot R_{\la -c_p}^{\vee}$. 
Thus, the loop $\ell_{p0}$ induces a linear map 
$\ell_{p0*}:H_1(M;\CL_{\la}^{\vee})\to H_1(M;\CL_{\la}^{(p0)\vee})$.
Since we have $T(u)|_{t_p \to t_p +1}=e^{-\pi \sqrt{-1} c_p} T(u)$, 
the monodromy structure of $\CL_{\la}^{(p0)\vee}$ as a local system on $M$ 
is same as that of 
$\CL_{\la}^{\vee}|_{\la \to \la -c_p}$. 
Note that the parameter $c_{\infty}$ for $\CL_{\la}^{\vee}|_{\la \to \la -c_p}$ is 
$c_{\infty}'=-(\la -c_p)-c_0\tau -c_1 t_1 -\cdots -c_n t_n =c_{\infty}+c_p$. 

For a twisted cycle $\ga \in Z_1 (M;\CL_{\la}^{\vee})$, 
we can construct a twisted cycle $\ga'\in Z_1 (M;\CL_{\la}^{(p0)\vee})$ in the same manner as $\ga$. 
The ``connection problem'' in this paper is to express $\ell_{p0*}([\ga])$'s 
in terms of $[\de']$'s. 

\begin{Lem}\label{lem:eigen-ellj0}
  We have
  \begin{align*}
    \ell_{p0*}([\ga_{j0}])=
    \begin{cases}
      [\ga_{j0}'] & (j<p) \\
      e^{\tpi (c_0-c_1-\dots -c_{p-1})}[\ga_{p0}'] & (j=p) \\
      e^{\tpi c_p}[\ga_{j0}'] & (j>p) .
    \end{cases}
  \end{align*}
\end{Lem}
\begin{proof}
  It is clear that the support of $\ell_{p0*}([\ga_{j0}])$ is same as that of $[\ga_{j0}']$. 
  Thus, it is sufficient to compare the branches at the point $P_0^{(j)}$. 
  This is not so difficult. 
\end{proof}

\begin{Th}
  Assume 
  $c_j,c_0-c_1-\cdots -c_{j} \not\in \Z$ ($j=1,\dots ,n$). 
  The linear map $\ell_{p0*}$ is expressed as 
  \begin{align*}
    \ell_{p0*}([\ga])
    &=\sum_{j=1}^{p-1} \frac{I_h ([\ga] ,[\ga_{j0}^{\vee}])}{I_h ([\ga_{j0}] ,[\ga_{j0}^{\vee}])}[\ga_{j0}'] \\
    &+e^{\tpi (c_0-c_1-\dots -c_{p-1})}\frac{I_h ([\ga] ,[\ga_{p0}^{\vee}])}{I_h ([\ga_{p0}] ,[\ga_{p0}^{\vee}])}[\ga_{p0}']
    +e^{\tpi c_p}\sum_{j=p+1}^{n} \frac{I_h ([\ga] ,[\ga_{j0}^{\vee}])}{I_h ([\ga_{j0}] ,[\ga_{j0}^{\vee}])}[\ga_{j0}'] \\
    &=-\sum_{j=1}^{p-1} 
      \frac{e^{\tpi (c_0+c_1+\cdots +c_{j-1})}(1-e^{\tpi c_j})}
      {(e^{\tpi c_0}-e^{\tpi (c_1+\cdots +c_{j-1})})(e^{\tpi c_0}-e^{\tpi (c_1+\cdots +c_j)})}
      I_h ([\ga] ,[\ga_{j0}^{\vee}])[\ga_{j0}'] \\
    &-\frac{e^{4\pii c_0}(1-e^{\tpi c_p})}
      {(e^{\tpi c_0}-e^{\tpi (c_1+\cdots +c_{p-1})})(e^{\tpi c_0}-e^{\tpi (c_1+\cdots +c_p)})}
      I_h ([\ga] ,[\ga_{p0}^{\vee}])[\ga_{p0}']\\
    &-e^{\tpi c_p}\sum_{j=p+1}^{n} 
      \frac{e^{\tpi (c_0+c_1+\cdots +c_{j-1})}(1-e^{\tpi c_j})}
      {(e^{\tpi c_0}-e^{\tpi (c_1+\cdots +c_{j-1})})(e^{\tpi c_0}-e^{\tpi (c_1+\cdots +c_j)})}
      I_h ([\ga] ,[\ga_{j0}^{\vee}])[\ga_{j0}'].
  \end{align*}
\end{Th}
\begin{proof}
  By Proposition \ref{prop:homology-int-diag-0infty}, any element $[\ga]\in H_1(M;\CL_{\la}^{\vee})$
  can be expressed as 
  \begin{align*}
    [\ga]=\sum_{j=1}^{n} \frac{I_h ([\ga] ,[\ga_{j0}^{\vee}])}{I_h ([\ga_{j0}] ,[\ga_{j0}^{\vee}])}[\ga_{j0}] .
  \end{align*}
  Thus, Lemma \ref{lem:eigen-ellj0} yields the theorem. 
\end{proof}
As a corollary, we obtain the representation matrix with respect to the basis 
$\{ [\ga_{1j}] \}_{j= 2,\dots,n-1,0,\infty}$ and $\{ [\ga_{1j}'] \}_{j= 2,\dots,n-1,0,\infty}$. 
We use the intersection matrices $H_{00}$ defined in Proposition \ref{prop:homology-int-diag-0infty}, 
$H_{1n}$ defined in Proposition \ref{prop:homology-int-diag-1n}, and 
$H_{10}=(I_h([\ga_{1j}],[\ga_{k0}^{\vee}]))_{\substack{j=2,\dots ,n-1 ,0,\infty \\ k=1,\dots ,n}}$, 
$H_{0n}=(I_h([\ga_{j0}],[\ga_{nk}^{\vee}]))_{\substack{j=1,\dots ,n \\ k=2,\dots ,n-1 ,\infty,0}}$.
\begin{Cor}[\cite{Ghazouani-Pirio}, \cite{Mano}]
  For a matrix $H$ whose entries belong to $\C(e^{\tpi c_{*}})$, 
  we set $H^{(p0)}=H|_{c_{\infty}\to c_{\infty}+c_p}$. 
  The matrix defined by
  \begin{align*}
    M_{p0}=(H_{1n}^{-1} \tp{H_{0n}})^{(p0)} \cdot 
    \diag (1,\dots ,1 ,\overset{\text{$p$-th}}{e^{\tpi (c_0-c_1-\dots -c_{p-1})}} ,e^{\tpi c_p} ,\dots ,e^{\tpi c_p})
    \cdot H_{00}^{-1} \tp{H_{10}} 
  \end{align*}
  satisfies the following equalities: 
  \begin{align}
    \label{eq:rep-mat-ellp0}
    &\ell_{p0*}([\ga_{12}],\dots ,[\ga_{1,n-1}],[\ga_{10}],[\ga_{1\infty}])
    =([\ga_{12}'],\dots ,[\ga_{1,n-1}'],[\ga_{10}'],[\ga_{1\infty}'])M_{p0}, \\
    \label{eq:relation-mat-ellp0}
    &H_{11}=\tp{M_{p0}} \cdot H_{11}^{(p0)} \cdot M_{p0}^{\vee}.
  \end{align}
\end{Cor}
\begin{proof}
  The diagonal matrix in the definition of $M_{p0}$ is nothing but the representation matrix 
  of $\ell_{p0*}$ with respect to the basis 
  $\{ [\ga_{0j}] \}_{j= 1,\dots,n}$ and $\{ [\ga_{0j}'] \}_{j= 1,\dots,n}$. 
  Since $H_{1n}$ and $H_{00}$ are diagonal, it is easy to see that 
  \begin{align*}
    ([\ga_{10}],\dots ,[\ga_{n0}])=([\ga_{12}],\dots ,[\ga_{1,n-1}],[\ga_{10}],[\ga_{1\infty}])H_{1n}^{-1} \tp{H_{0n}}
  \end{align*}
  and $(H_{1n}^{-1} \tp{H_{0n}})^{-1}=H_{00}^{-1} \tp{H_{10}}$. 
  We thus obtain the equality (\ref{eq:rep-mat-ellp0}). 
  Next, we show (\ref{eq:relation-mat-ellp0}). 
  Let $I_h^{(p0)}$ be the intersection form defined on $H_1(M;\CL_{\la}^{(p0)\vee})\times H_1(M;\CL_{\la}^{(p0)})$, 
  and $\ell_{p0*}^{\vee}:H_1(M;\CL_{\la})\to H_1(M;\CL_{\la}^{(p0)})$ be the linear map 
  induced by $\ell_{p0}$. 
  By definition of the intersection form, we have 
  $I_h([\ga],[\de])=I_h^{(p0)}(\ell_{p0*}([\ga]),\ell_{p0*}^{\vee}([\de]))$. 
  This implies the equality (\ref{eq:relation-mat-ellp0}). 
\end{proof}
Explicit expressions of $H_{10}$ and $H_{0n}$ are given in \S \ref{subsubsec:homology-int-matrices}. 
Since the inverse matrices of $H_{1n}$ and $H_{00}$ are easily obtained, 
each matrix in the definition of $M_{p0}$ has an explicit formula. 

\begin{Rem}
  On the level of integrals, the above discussion is interpreted as follows. 
  We set $p=1$ for simplicity. 
  By deforming a twisted cycle $\ga$, the integral (\ref{eq:RW-integral}) changes into 
  \begin{align*}
    &\int_{\ell_{10*}(\ga)} e^{\tpi c_0 u}\vth_1(u-t_1-1)^{c_1}\vth_1(u-t_2)^{c_2} \cdots \vth_1(u-t_n)^{c_n} 
    \frs(u-t'_j;\la-c_1)du \\
    &=e^{-\pi \sqrt{-1} c_1} 
      \int_{\ell_{10*}(\ga)} e^{\tpi c_0 u}\vth_1(u-t_1)^{c_1}\vth_1(u-t_2)^{c_2} \cdots \vth_1(u-t_n)^{c_n} 
      \frs(u-t'_j;\la-c_1)du , 
  \end{align*}
  where $t_j'=t_j+\de_{1j}$. 
  We rewrite the deformed cycle $\ell_{10*}(\ga)$ by using twisted cycles whose coefficients
  are in the local system defined by this integrand. 
  To avoid the constant $e^{-\pi \sqrt{-1} c_1}$, Mano uses the integrand $\tilde{\Phi}_j(w)$ in \cite{Mano} 
  instead of $T(u)\frs(u-t_j;\la)$. 
  For $\ell_{j\infty*}$ in \S \ref{subsubsec:connection+tau}, 
  we can also give a similar interpretation.
\end{Rem}

\subsubsection{Translating $t_p$ to $t_p +\tau$}\label{subsubsec:connection+tau}
For $p=1,\dots ,n$, let $\ell_{p\infty}$ be a path  
\begin{align*}
  [0,1]\ni s \mapsto (t_1^{\circ},\dots ,t_{p-1}^{\circ},t_{p}^{\circ}+s\tau,t_{p+1}^{\circ},\dots ,t_n^{\circ})\in \FT
\end{align*}
from $\bft^{\circ}$ to 
$\bft^{\circ}_{p\infty}=(t_1^{\circ},\dots ,t_{p-1}^{\circ},t_{p}^{\circ}+\tau,t_{p+1}^{\circ},\dots ,t_n^{\circ})$
(see Figure \ref{fig:path-p0}). 
Similarly to \S \ref{subsubsec:connection+1}, 
by considering the constraint 
$\la +c_0 \tau +c_1 t_1 +\cdots +c_n t_n +c_{\infty}=0$, 
we can obtain a linear map
$\ell_{p\infty *}:H_1(M;\CL_{\la}^{\vee})\to H_1(M;\CL_{\la}^{(p\infty)\vee})$, 
where $\CL_{\la}^{(p\infty)\vee} :=\C T(u)|_{t_p \to t_p +\tau} \ot R_{\la -c_p \tau}^{\vee}$. 
Since we have $T(u)|_{t_p \to t_p +\tau}=e^{\pi \sqrt{-1} c_p(\tau -2t_p+1)}\cdot e^{\tpi c_p u} T(u)$, 
the monodromy structure of $\CL_{\la}^{(p\infty)\vee}$ as a local system on $M$ 
is same as that of 
$\CL_{\la}^{\vee}|_{(c_0 ,\la) \to (c_0+c_p ,\la -c_p \tau )}$. 
Note that the parameter $c_{\infty}$ is not changed by this replacing.

For a twisted cycle $\ga \in Z_1 (M;\CL_{\la}^{\vee})$, 
we also write $\ga'\in Z_1 (M;\CL_{\la}^{(p\infty)\vee})$ for a twisted cycle constructed 
in the same manner as $\ga$. 
The following lemma, theorem, and corollary can be shown similarly to \S \ref{subsubsec:connection+1}. 

\begin{Lem}\label{lem:eigen-elljinfty}
  We have
  \begin{align*}
    \ell_{p\infty *}([\ga_{j\infty}])=
    \begin{cases}
      [\ga_{j\infty}'] & (j<p) \\
      e^{\tpi (-c_{\infty}+c_1+\dots +c_{p-1})}[\ga_{p\infty}'] & (j=p) \\
      e^{-\tpi c_p}[\ga_{j\infty}'] & (j>p) .
    \end{cases}
  \end{align*}
\end{Lem}

\begin{Th}
  Assume 
  $c_j,c_{\infty}-c_1-\cdots -c_{j} \not\in \Z$ ($j=1,\dots ,n$). 
  The linear map $\ell_{p \infty *}$ is expressed as 
  \begin{align*}
    \ell_{p\infty *}([\ga])
    &=\sum_{j=1}^{p-1} \frac{I_h ([\ga] ,[\ga_{j\infty}^{\vee}])}
      {I_h ([\ga_{j\infty}] ,[\ga_{j\infty}^{\vee}])}[\ga_{j\infty}'] \\
    &+e^{\tpi (-c_{\infty}+c_1+\dots +c_{p-1})}\frac{I_h ([\ga] ,[\ga_{p\infty}^{\vee}])}
      {I_h ([\ga_{p\infty}] ,[\ga_{p\infty}^{\vee}])}[\ga_{p\infty}']
    +e^{-\tpi c_p}\sum_{j=p+1}^{n} \frac{I_h ([\ga] ,[\ga_{j\infty}^{\vee}])}
      {I_h ([\ga_{j\infty}] ,[\ga_{j\infty}^{\vee}])}[\ga_{j\infty}'] \\
    &=-\sum_{j=1}^{p-1} 
      \frac{e^{\tpi (c_{\infty}+c_1+\cdots +c_{j-1})}(1-e^{\tpi c_j})}
      {(e^{\tpi c_{\infty}}-e^{\tpi (c_1+\cdots +c_{j-1})})(e^{\tpi c_{\infty}}-e^{\tpi (c_1+\cdots +c_{j})})}
      I_h ([\ga] ,[\ga_{j\infty}^{\vee}])[\ga_{j\infty}'] \\
    &-\frac{e^{4\pii (c_1+\dots +c_{p-1})}(1-e^{\tpi c_p})}
      {(e^{\tpi c_{\infty}}-e^{\tpi (c_1+\cdots +c_{p-1})})(e^{\tpi c_{\infty}}-e^{\tpi (c_1+\cdots +c_{p})})}
      I_h ([\ga] ,[\ga_{p\infty}^{\vee}])[\ga_{p\infty}']\\
    &-e^{-\tpi c_p}\sum_{j=p+1}^{n} 
      \frac{e^{\tpi (c_{\infty}+c_1+\cdots +c_{j-1})}(1-e^{\tpi c_j})}
      {(e^{\tpi c_{\infty}}-e^{\tpi (c_1+\cdots +c_{j-1})})(e^{\tpi c_{\infty}}-e^{\tpi (c_1+\cdots +c_{j})})}
      I_h ([\ga] ,[\ga_{j\infty}^{\vee}])[\ga_{j\infty}'].
  \end{align*}
\end{Th}
We set 
$H_{1\infty}=(I_h([\ga_{1j}],[\ga_{k\infty}^{\vee}]))_{\substack{j=2,\dots ,n-1 ,0,\infty \\ k=1,\dots ,n}}$, 
$H_{\infty n}=(I_h([\ga_{j\infty}],[\ga_{nk}^{\vee}]))_{\substack{j=1,\dots ,n \\ k=2,\dots ,n-1 ,\infty,0}}$, 
explicit expressions of which are given in \S \ref{subsubsec:homology-int-matrices}. 
Recall that $H_{1n}$ and $H_{\infty \infty}$ are diagonal matrices. 
\begin{Cor}[\cite{Ghazouani-Pirio}, \cite{Mano}]
  For a matrix $H$ whose entries belong to $\C(e^{\tpi c_{*}})$, 
  we set $H^{(p\infty)}=H|_{c_0\to c_0+c_p}$. 
  The matrix defined by
  \begin{align*}
    M_{p\infty}=(H_{1n}^{-1} \tp{H_{\infty n}})^{(p\infty)} \cdot 
    \diag (1,\dots ,1 ,\overset{\text{$p$-th}}{e^{\tpi (-c_{\infty}+c_1+\dots +c_{j-1})}} ,e^{-\tpi c_p} ,\dots ,e^{-\tpi c_p})
    \cdot H_{\infty \infty}^{-1} \tp{H_{1\infty}} 
  \end{align*}
  satisfies the following equalities: 
  \begin{align*}
    &\ell_{p\infty *}([\ga_{12}],\dots ,[\ga_{1,n-1}],[\ga_{10}],[\ga_{1\infty}])
    =([\ga_{12}'],\dots ,[\ga_{1,n-1}'],[\ga_{10}'],[\ga_{1\infty}'])M_{p\infty}, \\
    &H_{11}=\tp{M_{p\infty}} \cdot H_{11}^{(p\infty)} \cdot M_{p\infty}^{\vee}.
  \end{align*}
\end{Cor}

\section{Intersection theory for twisted cohomology group}\label{sec:cohomology-intersection}
As in \S \ref{subsec:homology-cohomology}, we assume $\la\in P=\{ a +b\tau \mid 0\leq a,b < 1 \}$ 
when we discuss the twisted cohomology groups. 
\subsection{Intersection form}
The cohomology intersection form $I_c$ is 
a non-degenerate bilinear form between $H^1(M;\CL_{\la})$ and $H^1(M;\CL_{\la}^{\vee})$: 
\begin{align*}
  I_c(\bu ,\bu) : H^1(M;\CL_{\la}) \times H^1(M;\CL_{\la}^{\vee}) \longrightarrow \C .
\end{align*}
By Fact \ref{fact:vanish-dim}, 
we have $H^1(M;\CL_{\la}^{\vee})\simeq \Om^1_{-\la}(*D)(E)/\na^{\vee}(\CO_{-\la}(*D)(E))$,
where $\na^{\vee} f =df -f\om$. 
Thus, we also identify $H^1(M;\CL_{\la}^{\vee})$ with $\Om^1_{-\la}(*D)(E)/\na^{\vee}(\CO_{-\la}(*D)(E))$. 

To define the intersection form, we introduce two de Rham cohomology groups. 
Let $\CE^i_{M}$ and $\CE^i_{M,c}$ be the sheaves of smooth $i$-forms on $M$ and 
those with compact support, respectively ($i=0,1,2$). 
We set $\CE^i_{\la}=\CE^i_{M}\ot_{\C} R_{\la}$ and $\CE^i_{\la,c}=\CE^i_{M,c}\ot_{\C} R_{\la}$.  
Then, $\na =d+\om \we $ is naturally defined on $\CE^{\bu}_{\la}$ and $\CE^{\bu}_{\la,c}$. 
There are two natural homomorphisms 
\begin{align*}
  \iota_1: H^1(M;\CL_{\la}) \longrightarrow H^1 (\CE^{\bu}_{\la}(M),\na),\qquad 
  \iota_2: H^1(\CE^{\bu}_{\la,c}(M),\na) \longrightarrow H^1 (\CE^{\bu}_{\la}(M),\na) .
\end{align*}
\begin{Prop}\label{prop:comparison}
  The morphisms $\iota_1$ and $\iota_2$ are isomorphisms. 
\end{Prop}
\begin{proof}
  It suffices to show that the following claims hold: 
  \begin{enumerate}[(1)]
  \item $\iota_1$ is surjective. 
  \item $\iota_2$ is surjective. 
  \item $H^0 (\CE^{\bu}_{\la}(M),\na)=0$ and $H^0 (\CE^{\bu}_{\la,c}(M),\na)=0$. 
  \item $H^2 (\CE^{\bu}_{\la}(M),\na)=0$ and $H^2 (\CE^{\bu}_{\la,c}(M),\na)=0$. 
  \end{enumerate}
  First, we temporarily admit these claims and prove the proposition. 
  The vanishing results (3), (4) imply 
  $\dim H^1 (\CE^{\bu}_{\la}(M),\na)=-\chi (M)=n$ (similarly, $\dim H^1 (\CE^{\bu}_{\la,c}(M),\na)=n$), 
  and hence $\iota_1$, $\iota_2$ are isomorphisms because of (1), (2). 
  Now, let us prove the claims. 
  We recall that for any $(0,1)$-form $\psi$ on $M$, there exists a function $f\in \CE^0_M(M)$
  such that $\bar{\pa} f =\psi$ (e.g., \cite[Theorem 25.6]{Forster}).
  Let $S$ be a holomorphic section of $\CO_E \ot R_{\la}$ on $M$ which has no zeros (cf. \cite[Lemma 30.2]{Forster}). 
  \begin{enumerate}[(1)]
  \item 
    Suppose $\psi =f_1 du +f_2 d\bar{u} \in \CE^1_{\la}(M)$ satisfies $\na \psi =0$. 
    Since $f_2 \in \CE^0_{\la}(M)$, we have $f_2/S \in \CE^0_M(M)$, and hence
    $(f_2/S)d\bar{u}$ is a $(0,1)$-form on $M$. 
    Thus, there exists a function $g\in \CE^0_M(M)$ such that $\bar{\pa} g =(f_2/S)d\bar{u}$. 
    Since $S$ is holomorphic, we have $f_2 d\bar{u} =\bar{\pa} (gS)$. 
    We set $\psi_1 =\psi -\na (gS)$. 
    Note that $[\psi_1]=[\psi]$ in $H^1 (\CE^{\bu}_{\la}(M),\na)$ because of $gS \in \CE^0_{\la}(M)$.  
    The $(1,0)$-form $\psi_1 =f_1 du -\pa (gS) -gS \om$ is a holomorphic one. 
    Indeed, we have 
    \begin{align*}
      \bar{\pa}\psi_1
      =\bar{\pa}(f_1 du) +\pa \bar{\pa}(gS) +\om \we \bar{\pa}(gS)
      =\bar{\pa}(f_1 du) +\pa (f_2 d\bar{u}) +\om \we (f_2 d\bar{u})
      =\na \psi =0 .
    \end{align*}
    By \cite[Proposition 2.5]{Mano-Watanabe}, there exists $\psi_2 \in \Om^1_{\la}(*D)(E)$ such that 
    $\psi_1 -\psi_2 \in \na (\CO_{\la}(M))$. 
    This implies $[\psi]=\iota_1([\psi_2])$, and hence the surjectivity of $\iota_1$ is proved. 
  \item
    It is sufficient to show that 
    for any $\vph \in \CE^1_{\la}(M)$ satisfying $\na \vph =0$, 
    there exist $\vph_1 \in \CE^1_{\la,c}(M)$ and $f\in \CE^0_{\la}(M)$ such that $\vph -\na f= \vph_1$. 
    We will find such $\vph_1$ and $f$ in a similar way to that in \cite[Proposition 3]{M-MA2019}. 
    For $j=1,\dots ,n$, let $U_j$ and $V_j$ be open neighborhoods of $t_j$ such that 
    $\overline{U_j}\subset V_j$ and $V_j \cap V_{k}=\emptyset$ ($j\neq k$), 
    and let $h_j$ be a smooth function on $E$ that satisfies $h_j \equiv 1$ on $U_j$ and 
    $h_j \equiv 0$ on $V_j^c$. For $v\in V_j-\{t_j\}$, we set 
    \begin{align}
      \label{eq:local-solution-integral}
      f_j(v) =\frac{1}{(e^{\tpi c_j}-1)T(v)} \int_{C_j(v)} T(u)\vph ,
    \end{align}
    where $C_j (v)$ is a loop in $V_j$ with terminal $v$ turning once around $t_j$ positively
    (a similar construction is also provided in \cite[Theorem 2.1]{Ito2010}). 
    By $\na \vph =0$ and Stokes' theorem, (\ref{eq:local-solution-integral}) is independent of the choice of the loop. 
    It is not difficult to see that $h_j f_j \in \CE^0_{\la}(M)$ and $\na(h_j f_j)=\vph$ on $U_j$. 
    We set $f=\sum_{j=1}^n h_j f_j$ and $\vph_1 =\vph -\na f$. 
    On each $U_j$, we have $\vph_1 =\vph -\na (h_j f_j)=0$, and hence 
    $\vph_1$ is an element in $\CE^1_{\la,c}(M)$. 
  \item
    Since the global solution to $\na f=0$ ($f\in \CE^0_{\la}(M)$ or $f\in \CE^0_{\la,c}(M)$) is only zero, 
    the $0$-th cohomology groups vanish. 
  \item
    First, we show $H^2 (\CE^{\bu}_{\la}(M),\na)=0$. 
    Let us consider $\eta =f du \we d\bar{u}\in \CE^2_{\la}(M)$. 
    Similarly to (1), there exists $g_1 \in \CE^0_{M}(M)$ such that $f d\bar{u} =\bar{\pa} (g_1 S)$. 
    We set $\psi =-g_1 S du$. Then we have $\psi \in \CE^1_{\la}(M)$ and 
    \begin{align*}
      \na \psi =-\bar{\pa} (g_1 S du) -\om \we (g_1 S du)
      =-f d\bar{u} \we du =\eta ,
    \end{align*}
    which implies $H^2 (\CE^{\bu}_{\la}(M),\na)=0$. 
    Next, we show $H^2 (\CE^{\bu}_{\la,c}(M),\na)=0$. 
    For any $\eta \in \CE^2_{\la,c}(M)$, we can find $\psi_1 \in \CE^1_{\la}(M)$ such that $\na \psi_1 =\eta$ 
    by the above discussion. 
    Since $\na \psi_1 =\eta$ has a compact support, we have $\na \psi_1 \equiv 0$ on a small neighborhood of each $t_j$. 
    By the same manner as (\ref{eq:local-solution-integral}), we can find $f_j$ satisfying 
    $\na f_j=\psi_1$ around $t_j$. 
    Similarly to (2), we can construct $f\in \CE^0_{\la}(M)$ such that 
    $\psi_2 =\psi_1 -\na f$ belongs to $\CE^1_{\la,c}(M)$. 
    Therefore we have $\eta =\na \psi_2 \in \na(\CE^1_{\la,c}(M))$, and 
    the proof is completed. 
    \qedhere
  \end{enumerate}
\end{proof}

\begin{Def}
  We set $\reg_c =\iota_2^{-1} \circ \iota_1 :H^1(M;\CL_{\la}) \to H^1(\CE^{\bu}_{\la,c},\na)$, 
  and call it the regularization map on the twisted cohomology group. 
\end{Def}
Thanks to the discussion in \cite{M-k-form}, we can evaluate the intersection number $I_c ([\vph],[\vph'])$ 
for $[\vph]\in H^1(M;\CL_{\la})$ and $[\vph']\in H^1(M;\CL_{\la}^{\vee})$ as follows. 
The intersection number is defined by 
\begin{align*}
  I_c ([\vph],[\vph']) = \int_{M} \vph_1 \we \vph' 
  \qquad (\vph_1 \in \CE^1_{\la,c}(M) \textrm{ satisfies } [\vph_1]=\reg_c([\vph])), 
\end{align*}
which converges and is well-defined. 
By using the expression $\vph_1=\vph -\na f=\vph -\na (\sum_{j=1}^n h_j f_j)$ 
in the proof of Proposition \ref{prop:comparison} (2), we obtain a formula 
\begin{align}
  \label{eq:intersecion-residue}
  I_c ([\vph],[\vph']) =\tpi \sum_{j=1}^n \Res_{u=t_j} (f_j \vph' ) 
\end{align}
in a similar manner to \cite{M-k-form}. 
Since $\vph \in \Om^1_{\la}(*D)(E)$ and the expression (\ref{eq:local-solution-integral}) of $f_j$
imply $f_j \in \CO_{\la}(*D)(V_j)$, we have $f_j \vph' \in \Om^1_E(*D)(V_j)$. 
Thus, the residue in (\ref{eq:intersecion-residue}) can be evaluated by using the 
Laurent expansions of $f_j$ and $\vph'$. 

Similarly to \S \ref{subsec:homology-intersection-form}, 
we often consider $c_j$'s and $\la$ as indeterminates and we can regard an intersection number 
as an element in the field 
$K(c_{*},\la,t_{*})$ of functions in $c_1,\dots ,c_n,c_0,\la$, $t_1,\dots ,t_n$ 
which has an involution
$(c_1,\dots,c_n,c_0,\la) \mapsto (-c_1,\dots,-c_n,-c_0,-\la)$. 
For $\vph(u;\la) \in\Om^1_{\la}(*D)(E)$, we set $\vph(u;\la)^{\vee}=\vph(u;-\la) \in \Om^1_{-\la}(*D)(E)$. 
Thus, for a $K(c_{*},\la,t_{*})$-linear combination $\vph=\sum_i a_i \vph_i$ 
($a_i\in K(c_{*},\la,t_{*})$, $\vph_i \in \Om^1_{\la}(*D)(E)$), 
we can naturally define $\vph^{\vee}\in \Om^1_{-\la}(*D)(E)$ by 
$\vph^{\vee}=\sum_i a_i^{\vee} \vph_i^{\vee}$. 
For example, we have $(c_1 \frs(u-t_1;\la)du)^{\vee}=-c_1 \frs(u-t_1;-\la)du$. 
By using these notations, 
we have $I_c([\vph],[\psi^{\vee}])=-I_c([\psi],[\vph^{\vee}])^{\vee}$, 
for $[\vph],[\psi]\in H^1(M;\CL_{\la})$. 

\subsection{Intersection numbers}
In this section, we give formulas of the intersection numbers for 
the twisted cocycles introduced in \S \ref{subsec:homology-cohomology}. 
We also define other twisted cocycles which will be used in \S \ref{subsec:contiguity}, 
and give their intersection numbers. 
Precise computations will be given in \S \ref{subsec:eval-cohomology-intersection}. 

\subsubsection{The case when $\la \in P-\{ 0\}$}\label{subsubsec:cohomology-lambda-nonzero}
We set $\psi_j =\frs (u-t_j ;\la)du \in \Om^1_{\la}(*D)(E)$
which has a simple pole at $u=t_j$ and satisfies $\Res_{u=t_j}(\psi_j)=1$.  
By Fact \ref{fact:cohomology-basis} (i), $\{ [\psi_j] \}_{j=1,\dots ,n}$ form a basis of 
$H^1(M;\CL_{\la})$. 
\begin{Th}\label{th:cohomology-intersection-1}
  We have
  \begin{align*}
    I_c ([\psi_j],[\psi_j^{\vee}])=\frac{\tpi}{c_j} ,\qquad 
    I_c ([\psi_j],[\psi_k^{\vee}])=0 \quad (j\neq k).
  \end{align*}
\end{Th}
Note that the determinant of the intersection matrix 
$C_{\psi \psi}=\big( I_c ([\psi_j],[\psi_k^{\vee}]) \big)_{j,k=1,\dots ,n}$ 
is equal to $(\tpi)^n/(c_1\cdots c_n)\neq 0$. 
Thus, by using the intersection form, we can also verify that $\{ [\psi_j] \}_{j=1,\dots ,n}$ form a basis. 

We set $\phi_p =\frac{\pa \frs}{\pa u}(u-t_p;\la) du\in \Om^1_{\la}(*D)(E)$
which has a pole of order $2$ at $u=t_p$. 
For the discussion in \S \ref{subsec:contiguity}, we show that 
$\{ [\phi_p] \} \cup \{ [\psi_k] \}_{k\neq p}$ also form a basis, and 
we construct another basis dual to it. 
\begin{Prop}\label{prop:cohomology-intersection-2}
  For $p,q\in \{ 1,\dots ,n \}$ with $p\neq q$, we set 
  \begin{align*}
    \vph^{(pq)}_{j}&=
    \begin{cases}
      \phi_p & (j=q) \\
      \psi_j & (j\neq q), 
    \end{cases}
    \\
    \eta^{(pq)}_{k}&=
    \begin{cases}
      \psi_p+\frac{1}{\frs(t_p-t_q;\la)}\Big( 
      \frac{1}{c_p}\Big(\tpi c_0 +\sum_{l\neq p} c_l \rho(t_p -t_l) \Big) -\rho(-\la) \Big)\psi_q & (k=p) \\
      \psi_q & (k=q) \\
      \psi_k -\frac{\frs (t_p-t_k;\la)}{\frs (t_p-t_q;\la)}\psi_q & (k\neq p,q). 
    \end{cases}
  \end{align*}
  Then we have $I_c ([\vph^{(pq)}_{j}], [\eta^{(pq)\vee}_{k}])=0$ if $j\neq k$, and 
  \begin{align*}
    I_c ([\vph^{(pq)}_{q}],[\eta^{(pq)\vee}_{q}])
    &=I_c ([\phi_p],[\psi_q^{\vee}])
      =-\tpi \cdot \frac{\frs(t_p-t_q;-\la)}{c_p-1},\\
    I_c ([\vph^{(pq)}_{j}],[\eta^{(pq)\vee}_{j}])
    &=I_c ([\psi_j],[\psi_j^{\vee}])=\frac{\tpi}{c_j} \quad (j\neq q).
  \end{align*}
\end{Prop}
We have constructed $\{\eta^{(pq)}_{k}\}_{k=1,\dots ,n}$ so that 
the intersection matrix $C_{\vph \eta}=\big( I_c([\vph^{(pq)}_j],[\eta_k^{(pq)\vee}])\big)_{j,k=1,\dots,n}$ becomes diagonal. 
For details, see \S \ref{subsec:eval-cohomology-intersection}. 
Note that $\eta^{(pq)\vee}_{k}$ can be defined under a condition $t_p-t_q+\la \not\in \La_{\tau}$. 
We thus conclude that $\{ [\vph^{(pq)}_{j} ]\}_{j=1,\dots ,n}=\{ [\phi_p] \} \cup \{ [\psi_k] \}_{k\neq p}$ form
a basis of $H^1(M;\CL_{\la})$ under this condition.

\subsubsection{The case when $\la =0$}
As mentioned in Fact \ref{fact:cohomology-basis} (ii), 
to obtain a basis of $H^1(M;\CL_{\la})$ for $\la=0$, 
we need to use a $1$-form having a pole of order $2$. 
Though $t_1$ is specified in Fact \ref{fact:cohomology-basis} (ii), 
there is no difficulty caused by specifying another $t_i$. 
Thus, in this section, we consider 
\begin{align*}
  \vph_0 =du ,\quad 
  \vph_{ii} =\rho' (u-t_i)du ,\quad  
  \vph_{ij} =(\rho (u-t_j) -\rho (u-t_i))du \quad (j\in \{1,\dots ,n\}-\{ i \}). 
\end{align*}
We often write $\vph_{i0}$ for $\vph_0$. 
\begin{Th}\label{th:cohomology-intersection-0}
  For $j,k \in \{1,\dots ,n\} -\{ i \}$, we have 
  \begin{align*}
    I_c ([\vph_0],[\vph_{ij}^{\vee}])
    &=I_c ([\vph_{ij}],[\vph_0^{\vee}])=0, \quad 
    I_c ([\vph_{ij}],[\vph_{ik}^{\vee}])=\tpi \left( \frac{1}{c_i}+\frac{\de_{jk}}{c_j} \right) ,\\
    I_c ([\vph_{ii}],[\vph_{0}^{\vee}])
    &=-\frac{\tpi}{c_i-1}, \\
    I_c ([\vph_{ii}],[\vph_{ii}^{\vee}])
    &=\frac{\tpi}{(c_i-1)(c_i+1)} \Bigg( 
      \frac{1}{c_i}\Big( \tpi c_0 +\sum_{l\neq i} c_l \rho(t_i-t_l) \Big)^2 
      -c_i\frac{\vth'''_{1}(0)}{\vth'_{1}(0)} 
      -\sum_{l\neq i} c_l \rho'(t_i-t_l)
      \Bigg), \\
    I_c ([\vph_{ii}],[\vph_{ij}^{\vee}])
    &=-\frac{\tpi}{c_i(c_i-1)}\Big( \tpi c_0 +\sum_{\substack{l\neq i}} c_l \rho(t_i -t_l) 
      +c_i \rho(t_i -t_j) \Big), 
  \end{align*}
  and the other intersection numbers are obtained by applying the formula 
  $I_c([\vph],[\psi^{\vee}])=-I_c([\psi],[\vph^{\vee}])^{\vee}$. 
\end{Th}
Since the determinant of the intersection matrix 
$(I_c ([\vph_{1j}],[\vph_{1k}^{\vee}]))_{j,k=0,\dots ,n-1}$ 
is equal to 
\begin{align*}
  -(\tpi)^n\frac{c_1+\cdots +c_{n-1}}{(c_1-1)(c_1+1)c_1\cdots c_{n-1}}
  =\frac{(\tpi)^n c_n}{(c_1-1)(c_1+1)c_1\cdots c_{n-1}}, 
\end{align*}
we can verify that $\{ [\vph_{1j}] \}_{j=0,\dots ,n-1}$ form a basis of $H^1(M;\CL_{\la})$ 
under the condition $c_j \not\in \Z$. 
Similarly, $\{ [\vph_{ij}] \}_{j=0,\dots ,\check{k},\dots ,n}$ with $k\neq 0,i$ also form a basis. 

\subsection{Contiguity relations}\label{subsec:contiguity}
In the case when $\la \in P-\{ 0\}$, the contiguity relations for the Riemann-Wirtinger integral 
can be expressed in terms of intersection forms on twisted cohomology groups. 
Basic idea is same as that in \cite[Section 5]{GM-Pfaffian-contiguity}. 

We consider the shift of the parameters $(c_p ,c_q)\to (c_p +1,c_q-1)$. 
The corresponding local system, denoted by $\CL_{\la}^{(p+,q-)}$, is obtained by 
replacing $(c_p ,c_q,\la)$ by $(c_p +1,c_q-1,\la-t_p+t_q)$, 
because of the constraint $\la +c_0 \tau +c_1 t_1 +\cdots +c_n t_n +c_{\infty}=0$. 
Here, we assume $\la-t_p+t_q \not\in \La_{\tau}$ which is also implicitly assumed in \cite[Section 5]{Mano}. 

For $\vph \in \Om^1_{\la}(*D)(E)$, 
we define $\vph^{(p+,q-)}\in \Om^1_{\la-t_p+t_q}(*D)(E)$ by replacing
$(c_p ,c_q,\la)$ by $(c_p +1,c_q-1,\la-t_p+t_q)$. 
The twisted cohomology group $H^1(M;\CL_{\la}^{(p+,q-)})$ is identified with 
the de Rham cohomology group with respect to $\na^{(p+,q-)}$, where 
$\na^{(p+,q-)} \vph =\na \vph +d\log (\vth_1(u-t_p)/\vth_1(u-t_q))\we \vph$. 
Let $I_c^{(p+,q-)}$ denotes the intersection form on 
$H^1(M;\CL_{\la}^{(p+,q-)})\times H^1(M;\CL_{\la}^{(p+,q-)\vee})$. 
If we regard the parameters as indeterminates, we have 
$I_c^{(p+,q-)}([\vph^{(p+,q-)}],[\psi^{(p+,q-)\vee}])=I_c([\vph],[\psi^{\vee}])^{(p+,q-)}$, 
where the last ${}^{(p+,q-)}$ means replacing $(c_p ,c_q,\la)$ with $(c_p +1,c_q-1,\la-t_p+t_q)$, and 
$\psi^{(p+,q-)\vee}=(\psi^{(p+,q-)})^{\vee}$. 

By the same discussion of \cite[Proposition 5.2]{GM-Pfaffian-contiguity}, we obtain the following. 
\begin{Prop}
  The map defined by 
  \begin{align*}
    \CS^p_q :H^1(M;\CL_{\la}^{(p+,q-)}) \ni [\vph] \mapsto 
    \Big[\frac{\vth_1(u-t_p)}{\vth_1(u-t_q)}\cdot \vph \Big] \in H^1(M;\CL_{\la}) 
  \end{align*}
  is a well-defined linear map. 
\end{Prop}
For a fixed twisted cycle $\ga \in Z_1 (M;\CL_{\la}^{\vee})$, we set 
\begin{align*}
  \bff
  =\tp{\left( \int_{\ga}T(u)\psi_1 ,\dots ,\int_{\ga}T(u)\psi_n \right)}, 
\end{align*}
and $\bff^{(p+,q-)}$ denotes the vector obtained by 
replacing $(c_p ,c_q,\la)$ with $(c_p +1,c_q-1,\la-t_p+t_q)$. 
A relation between $\bff$ and $\bff^{(p+,q-)}$ is called the contiguity relation. 
For $\vph \in \Om^1_{\la}(*D)(E)$, we have 
\begin{align*}
  \int_{\ga} T(u) \cdot \left( \frac{\vth_1(u-t_p)}{\vth_1(u-t_q)}\vph^{(p+,q-)} \right)
  &=\int_{\ga} \left(T(u)|_{(c_p ,c_q)\to (c_p +1,c_q-1)}\right) \cdot \vph^{(p+,q-)}  \\
  &=\left. \Big( \int_{\ga}  T(u)\vph \Big)\right|_{(c_p ,c_q,\la)\to (c_p +1,c_q-1,\la-t_p+t_q)} .
\end{align*}
Therefore, if $S^p_q$ denotes the representation matrix of $\CS^p_q$ with respect to 
the bases $\{ [\psi_j^{(p+,q-)}] \}\subset H^1(M;\CL_{\la}^{(p+,q-)})$ and $\{ [\psi_j] \}\subset H^1(M;\CL_{\la})$, 
the contiguity relation is obtained as $\bff^{(p+,q-)}=S^p_q \cdot \bff$. 
We will give an explicit formula of $S^p_q$ in Corollary \ref{cor:rep-mat-contiguity}.

Since the expression of $\CS^p_q ([\psi_q^{(p+,q-)}])$ as a linear combination of $\{ [\psi_j] \}$ is complicated, 
we first consider the basis $\{[(\vph^{(pq)}_j)^{(p+,q-)}]\}\subset H^1(M;\CL_{\la}^{(p+,q-)})$ 
defined in \S \ref{subsubsec:cohomology-lambda-nonzero}. 
The image of this basis under $\CS^p_q$ can be obtained by the following lemma
which is proved by straightforward calculation. 
\begin{Lem}\label{lem:contiguity-basis}
  The following equalities hold as elements in $\Om^1_{\la}(*D)(E)$: 
  \begin{align}
    \label{eq:contiguity-basis-1}
    \frac{\vth_1(u-t_p)}{\vth_1(u-t_q)}\psi_j^{(p+,q-)}
    &=\frac{\vth_1 (t_j-t_p)}{\vth_1 (t_j-t_q)} \psi_j 
      +\frac{\vth_1 (t_p -t_q) \vth_1 (\la -t_p+t_j)}{\vth_1(t_j-t_q) \vth_1 (\la -t_p+t_q)} \psi_q 
    \quad (j\neq p,q),\\
    \label{eq:contiguity-basis-2}
    \frac{\vth_1(u-t_p)}{\vth_1(u-t_q)}\psi_p^{(p+,q-)}
    &=\frac{\vth_1(\la)}{\vth_1(\la -t_p +t_q)} \psi_q ,\\
    \label{eq:contiguity-basis-3}
    \frac{\vth_1(u-t_p)}{\vth_1(u-t_q)}\phi_p^{(p+,q-)}
    &=\frac{\vth_1(\la)}{\vth_1 (\la-t_p +t_q)} 
      (\rho(t_p -t_q)-\rho (\la) ) \psi_q 
      -\frac{\vth_1'(0)}{\vth_1 (t_p-t_q)} \psi_p.
  \end{align}
\end{Lem}
\begin{proof}
  The equality (\ref{eq:contiguity-basis-2}) is obvious. 
  The equalities (\ref{eq:contiguity-basis-1}) and (\ref{eq:contiguity-basis-3})
  follow from (\ref{eq:theta-rel-2}) and (\ref{eq:theta-rel-1}), respectively. 
  As an example, we show (\ref{eq:contiguity-basis-3}): 
  \begin{align*}
    &\frac{\vth_1(u-t_p)}{\vth_1(u-t_q)}\cdot \frac{\pa}{\pa u}\frs (u-t_p ;\la-t_p+t_q) \\
    &=\frac{(\vth_1'(u-\la -t_q) \vth_1(u-t_p)-\vth_1(u-\la -t_q) \vth_1'(u-t_p)) \vth_1'(0)}
      {\vth_1(u-t_p) \vth_1 (u-t_q) \vth_1 (-\la+t_p -t_q)} \\
    &=\frac{\vth_1'(0)}{\vth_1 (-\la+t_p -t_q)}
      \frac{\vth_1(u-\la -t_q)}{\vth_1 (u-t_q) } \left( 
      \frac{\vth_1'(u-\la -t_q)}{\vth_1 (u-\la-t_q)}
      -\frac{\vth_1'(u-t_p)}{\vth_1(u-t_p)}
      \right) \\
    &=\frac{\vth_1(-\la)}{\vth_1 (-\la+t_p -t_q)}
      \frs (u-t_q;\la) (\rho(u-\la -t_q)-\rho (u-t_p) ) \\
    &=\frac{\vth_1(-\la)}{\vth_1 (-\la+t_p -t_q)} \Big(
      \frs (u-t_q;\la) (\rho(t_p -t_q)-\rho (\la) )
      -\frs (u-t_p;\la) \frs (t_p-t_q;\la) 
      \Big)
      . \qedhere
  \end{align*}
\end{proof}
By using the basis $\{[\eta^{(pq)\vee}_{j}]\}_{j=1,\dots ,n}$ defined in \S \ref{subsubsec:cohomology-lambda-nonzero}, 
we can express $\CS^p_q$ in terms of the intersection form. 
\begin{Th}
  Suppose $t_p-t_q \pm \la \not\in \La_{\tau}$. 
  For any $\vph \in \Om^1_{\la}(*D)(E)$, we have 
  \begin{align}
    \nonumber
    &\CS^p_q (\vph^{(p+,q-)}) \\
    \label{eq:contiguity-expression-1}
    &=\sum_{j=1}^n \frac{I_c^{(p+,q-)}(\vph^{(p+,q-)},\eta^{(pq)(p+,q-)\vee}_j)}
      {I_c^{(p+,q-)}(\vph^{(pq)(p+,q-)}_j,\eta^{(pq)(p+,q-)\vee}_j)}
      \CS^p_q(\vph^{(pq)(p+,q-)}_j) \\
    \nonumber
    &=\sum_{j\neq p,q} \frac{I_c^{(p+,q-)}(\vph^{(p+,q-)},\eta^{(pq)(p+,q-)\vee}_j)}
      {\tpi}\frac{c_j \vth_1 (t_j-t_p)}{\vth_1 (t_j-t_q)} \psi_j 
      +\frac{I_c^{(p+,q-)}(\vph^{(p+,q-)},\eta^{(pq)(p+,q-)\vee}_q)}{\tpi} 
      \frac{c_p \vth_1(\la-t_p+t_q)}{\vth_1 (\la)} \psi_p\\
    \nonumber
    &\quad 
      +\Big(
      \frac{I_c^{(p+,q-)}(\vph^{(p+,q-)},\eta^{(pq)(p+,q-)\vee}_p)}{\tpi}
      \frac{(c_p+1)\vth_1(\la)}{\vth_1(\la -t_p +t_q)} \\
    \nonumber
    &\qquad \qquad 
      -\frac{I_c^{(p+,q-)}(\vph^{(p+,q-)},\eta^{(pq)(p+,q-)\vee}_q)}{\tpi}
      \frac{c_p \vth_1 (t_p -t_q)}{\vth_1'(0)} (\rho(t_p -t_q)-\rho (\la) ) \\
    \label{eq:contiguity-expression-2}
    &\qquad \qquad 
      +\sum_{j\neq p,q} \frac{I_c^{(p+,q-)}(\vph^{(p+,q-)},\eta^{(pq)(p+,q-)\vee}_j)}{\tpi}
      \frac{c_j \vth_1 (t_p-t_q) \vth_1 (\la -t_p+t_j)}{\vth_1(t_j-t_q) \vth_1 (\la -t_p+t_q)}
      \Big) \psi_q .
  \end{align}
  Here, we omit $[~]$ and use notations 
  $\vph^{(pq)(p+,q-)}_j=(\vph^{(pq)}_j)^{(p+,q-)}$ and 
  $\eta^{(pq)(p+,q-)\vee}_j=((\eta^{(pq)}_j)^{(p+,q-)})^{\vee}$
  for simplicity. 
\end{Th}
\begin{proof}
  If we set $\vph =\vph^{(pq)}_j$ in the right-hand side of (\ref{eq:contiguity-expression-1}), 
  then it coincides with $\CS^p_q(\vph^{(pq)(p+,q-)}_j)$. 
  Since $\{[(\vph^{(pq)}_j)^{(p+,q-)}]\}$ form a basis, 
  the equality (\ref{eq:contiguity-expression-1}) holds. 
  The expression (\ref{eq:contiguity-expression-2}) follows from 
  Proposition \ref{prop:cohomology-intersection-2} and 
  Lemma \ref{lem:contiguity-basis}. 
\end{proof}
Though our expression seems to be complicated, 
we can treat with not only $\psi_j$, but also 
any $\vph \in \Om^1_{\la}(*D)(E)$. 
Once we obtain the intersection numbers 
$I_c(\vph,\eta^{(pq)\vee}_j)$'s, 
we can express $\CS^p_q(\vph^{(p+,q-)})$ as a linear combination of $\psi_j$'s.  
\begin{Rem}
  In general, to derive the contiguity relations in terms of twisted cohomology groups, 
  we need to find some $f\in \CO_{\la}(*D)(E)$ such that 
  $\CS^p_q \vph -\na f$ becomes a linear combination of the basis $\{ [\psi_j] \}_{j=1,\dots ,n}$
  (e.g., \cite[Theorem 5.1]{Mano}).   
  Our method can be avoid this difficulty.   
\end{Rem}

Similarly to the connection problems, we can obtain the representation matrix of $\CS^p_q$. 
We use
the diagonal matrices $C_{\psi \psi}$ and $C_{\vph \eta}$ whose diagonal entries
are given in 
Theorem \ref{th:cohomology-intersection-1} and Proposition \ref{prop:cohomology-intersection-2}, 
and we
set 
$C_{\psi \eta}=(I_c([\psi_j],[\eta_k^{(pq)\vee}]))_{j,k=1,\dots,n}$.
By definition of $\eta_j^{(pq)}$, we have a simple relation 
\begin{align*}
  \tp (\eta_1^{(pq)},\dots ,\eta_n^{(pq)})
  =A \cdot \tp (\psi_1 ,\dots ,\psi_n), \quad A=\id_n+A',
\end{align*}
where 
$A'$ is an $n\times n$ matrix 
whose entries are zero except for the $(p,q)$-entry 
\begin{align*}
  \frac{1}{\frs(t_p-t_q;\la)}\Big(
  \frac{1}{c_p}\Big(\tpi c_0 +\sum_{k\neq p} c_k \rho(t_p -t_k) \Big) -\rho(-\la) \Big)
\end{align*}
and the $(j,q)$-entry $-\frac{\frs (t_p-t_j;\la)}{\frs (t_p-t_q;\la)}$
for $j\neq p,q$.
An explicit formula of $C_{\psi \eta}$ is given by $C_{\psi \eta}=C_{\psi \psi}\cdot \tp{A}^{\vee}$. 
\begin{Cor}[\cite{Mano}]\label{cor:rep-mat-contiguity}
  For a matrix $C$ whose entries belong to 
  $K(c_{*},\la,t_{*})$, 
  we set 
  $C^{(p+,q-)}=C|_{(c_p ,c_q,\la)\to (c_p +1,c_q-1,\la-t_p+t_q)}$. 
  Let $(S^p_q)'$ be the matrix satisfying 
  \begin{align*}
    \CS^p_q ~\tp ([\vph^{(pq)(p+,q-)}_1],\dots ,[\vph^{(pq)(p+,q-)}_n])
    =(S^p_q)'\cdot \tp ([\psi_1],\dots ,[\psi_n]), 
  \end{align*}
  the entries of which are given in  
  Lemma \ref{lem:contiguity-basis}. 
  The matrix defined by
  \begin{align*}
    S^p_q=(C_{\psi \eta} C_{\vph \eta}^{-1} )^{(p+,q-)} \cdot (S^p_q)'
  \end{align*}
  satisfies the following equalities: 
  \begin{align}
    \label{eq:rep-mat-contiguity}
    &\CS^p_q ~\tp ([\psi_1^{(p+,q-)}],\dots ,[\psi_n^{(p+,q-)}])
      =S^p_q\cdot \tp ([\psi_1],\dots ,[\psi_n]), \\
    \label{eq:relation-mat-contiguity}
    &S^p_q \cdot C_{\psi \psi}=(C_{\psi \psi}\cdot (\tp{S^p_q})^{\vee})^{(p+,q-)}.
  \end{align}
\end{Cor}
\begin{proof}
  Since we have 
  $\tp([\psi_1],\dots ,[\psi_n])=C_{\psi \eta} C_{\vph \eta}^{-1} \cdot \tp([\vph^{(pq)}_1],\dots ,[\vph^{(pq)}_n])$, 
  the equality (\ref{eq:rep-mat-contiguity}) is easily obtained. 
  The equality (\ref{eq:relation-mat-contiguity}) follows from the property
  $I_c(\CS^p_q ([\vph]), [\psi])=I_c^{(p+,q-)}([\vph], (\CS^p_q)^{\vee}([\psi]))$, where 
  $(\CS^p_q)^{\vee}:H^1(M;\CL_{\la}^{\vee})\ni [\vph]\mapsto 
  [(\vth_1(u-t_p)/\vth_1(u-t_q))\cdot \vph]\in H^1(M;\CL_{\la}^{(p+,q-)\vee})$. 
  This property can be shown in the same manner as 
  \cite[Proposition 5.8]{GM-Pfaffian-contiguity}. 
\end{proof}
\begin{Ex}
  As in \cite[Theorem 5.1]{Mano}, the expression of $\CS^p_q(\psi^{(p+,q-)}_q)$ is complicated. 
  According to \cite{Mano}, the coefficient of $\psi_q$ should be 
  \begin{align}
    \label{eq:example-contiguity-psi_q-1}
    \frac{\vth_1 (t_q -t_p)}{\vth_1'(0)} \Big( 
    \rho(t_q -t_p) -\rho(\la-t_p +t_q)
    +\frac{1}{1-c_q} \Big(\tpi c_0 +\sum_{j\neq q} c_j \rho(t_q -t_j) -c_q \rho (\la) \Big) 
    \Big) .
  \end{align}
  Let us verify that it coincides with the $(q,q)$-entry of $S^p_q$. 
  Since the entries of the $q$-th row of $C_{\psi \eta}$ are 
  \begin{align*}
    \frac{(C_{\psi \eta})_{qq}}{\tpi}
    &=\frac{1}{c_q} ,\qquad 
    \frac{(C_{\psi \eta})_{qj}}{\tpi} 
    =-\frac{1}{c_q} \frac{\vth_1(t_p-t_j+\la)\vth_1(t_p-t_q)}{\vth_1(t_p-t_j)\vth_1(t_p-t_q+\la)}
      \ (j\neq p,q), \\
    \frac{(C_{\psi \eta})_{qp}}{\tpi}
    &=\frac{1}{c_p} \frac{\vth_1(t_p -t_q)}{\vth_1'(0)}
      \Big( \frac{1}{c_q}\Big(\tpi c_0 +\sum_{j\neq p,q} c_j \rho(t_p -t_j) -c_p \rho(\la) \Big) 
      +\rho(t_p -t_q)  \Big)
      \frac{\vth_1(\la)}{\vth_1(t_p -t_q +\la)}, 
  \end{align*}
  the $(q,q)$-entry of $S^p_q$ is equal to 
  \begin{align}
    \nonumber
    &\frac{\vth_1 (t_p -t_q)}{\vth_1'(0)} \Big(
    -\frac{c_p}{c_q-1} \rho(t_p -t_q) +\frac{c_p}{c_q-1}\rho (\la)
      \\
    \nonumber
    &\quad 
      +\frac{1}{c_q-1} \Big(\tpi c_0
      +\sum_{j\neq p,q} c_j \rho(t_p -t_j) 
      -(c_p+1) \rho(\la-t_p +t_q) \Big)
      +\rho(t_p -t_q) \\
    \label{eq:example-contiguity-psi_q-2}
    &\quad 
      -\sum_{j\neq p,q} 
      \frac{c_j}{c_q-1}\cdot 
      \frac{\vth_1'(0)\vth_1 (\la-t_p +t_j)\vth_1(\la-t_j+t_q)\vth_1(t_p-t_q)}
      {\vth_1(\la)\vth_1 (\la-t_p+t_q)\vth_1(t_p-t_j)\vth_1(t_j-t_q) }
      \Big) .
  \end{align}
  By (\ref{eq:theta-rel-1}), we have 
  \begin{align*}
    \frac{\vth_1'(0)\vth_1 (\la-t_p +t_j)\vth_1(\la-t_j+t_q)\vth_1(t_p-t_q)}
    {\vth_1(\la)\vth_1 (\la-t_p+t_q)\vth_1(t_p-t_j)\vth_1(t_j-t_q) }
    =\rho(t_p-t_j)+\rho(t_j-t_q)-\rho(t_p-t_q-\la)-\rho(\la) . 
  \end{align*}
  Using this relation and $\sum_{j\neq p,q}c_j =-c_p -c_q$, we can show that 
  (\ref{eq:example-contiguity-psi_q-2}) coincides with (\ref{eq:example-contiguity-psi_q-1}). 
  Note that our computation does not require the relation 
  $[\na f]=0$ in $H^1(M;\CL_{\la})$ 
  for some $f\in \CO_{\la}(*D)(E)$. 
\end{Ex}

\section{Computation of intersection numbers}\label{sec:computation}
We give precise computations of the intersection numbers. 

\subsection{Intersection numbers of twisted cycles}\label{subsec:eval-homology-intersection}
Since the intersection numbers that we use in this paper are so many, 
we cannot explain all of them. 

\subsubsection{Fact \ref{fact:homology-intersection-basis} and Corollary \ref{cor:homology-intersection-interval}}
Fact \ref{fact:homology-intersection-basis} is computed in \cite[Proposition 3.4.1]{Ghazouani-Pirio}. 
A detailed computation of $I_h([\ga_{1\infty}],[\ga_{1\infty}^{\vee}])$ is given in \cite{Ghazouani-Pirio}.
For readers' convenience, we explain $I_h([\ga_{1j}],[\ga_{1\infty}^{\vee}])$ ($j=2,\dots ,n$) and 
$I_h([\ga_{10}],[\ga_{1\infty}^{\vee}])$ in detail. 

First of all, 
we consider the intersection numbers $I_h ([\ga_{jk}] ,[\ga_{j'k'}^{\vee}])$ 
for $j,j',k,k'\in \{1,\dots,n\}$ satisfying $j<k$, $j'<k'$. 
As mentioned above Corollary \ref{cor:homology-intersection-interval}, the cases of $j,j'\geq 2$ 
follow from that of $j,j'=1$. 
In fact, it is not difficult to compute all the cases directly. 
A method to compute is quite similar to \cite{KY}, see also \cite[Fact 4.2]{M-FD}.  
For example, we have 
\begin{align*}
  I_h ([\ga_{12}] ,[\ga_{23}^{\vee}])
  =\frac{-1}{e^{\tpi c_2}-1} \cdot (-1) \cdot 1
  =\frac{-1}{1-e^{\tpi c_2}}.
\end{align*}
Here, the first factor is the coefficient of $s_2$ in $\ga_{12}$, 
the second ``$(-1)$'' is the local intersection multiplicity, 
and the third ``$1$'' indicates the difference of the branches 
(see Figure \ref{fig:homology-intersection-12-23}). 
\begin{figure}
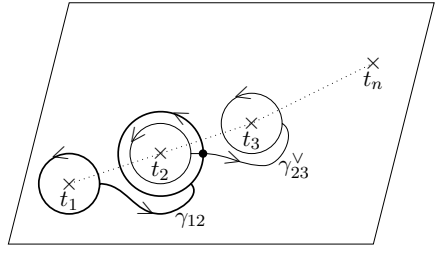

  \centering
  \scalebox{0.8}{\RWpictureH}
  \caption{The twisted cycles $\ga_{12}$ and $\ga_{23}^{\vee}$
    ($\bullet$ is the intersection point.)}
  \label{fig:homology-intersection-12-23}
\end{figure}

Next, we compute $I_h([\ga_{1j}],[\ga_{1\infty}^{\vee}])$ ($j=2,\dots ,n$). 
By Figure \ref{fig:homology-intersection-1j-1infty}, we can compute it as 
\begin{align*}
  I_h ([\ga_{1j}] ,[\ga_{1\infty}^{\vee}]) 
  =\frac{1-e^{\tpi c_{\infty}}}{e^{-\tpi c_1}-1} \cdot 1 \cdot 1
  =\frac{e^{\tpi c_1}(1-e^{\tpi c_{\infty}})}{1-e^{\tpi c_1}} .
\end{align*}
Note that since the coefficient of $m_0$ in $\ga_{1\infty}$ is $\frac{1-e^{-\tpi c_{\infty}}}{e^{\tpi c_1}-1}$, 
that in $\ga_{1\infty}^{\vee}$ is obtained by operating ${}^{\vee}$ to it.  
\begin{figure}
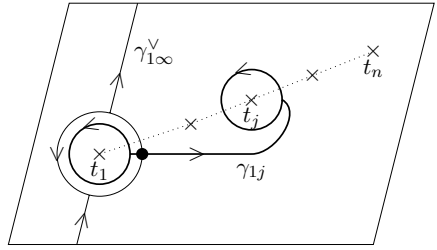

  \centering
  \scalebox{0.8}{\RWpictureI}
  \caption{The twisted cycles $\ga_{1j}$ and $\ga_{1\infty}^{\vee}$ 
    ($\bullet$ is the intersection point.)}
  \label{fig:homology-intersection-1j-1infty}
\end{figure}

Now, we compute $I_h([\ga_{10}],[\ga_{1\infty}^{\vee}])$. 
As in Figure \ref{fig:homology-intersection-10-1infty}, 
there are two intersection points $m_0 \cap l_{\infty}$ ($\bullet$) and $m_3 \cap l_{\infty}$ ($\circ$).  
Thus, the intersection number is computed as 
\begin{align*}
  I_h ([\ga_{10}] ,[\ga_{1\infty}^{\vee}]) 
  &=\frac{1-e^{\tpi c_0}}{e^{\tpi c_1}-1} \cdot (-1) \cdot e^{\tpi c_1}
    +\frac{e^{\tpi c_1}-e^{\tpi c_0}}{e^{\tpi c_1}-1} \cdot 1 \cdot e^{\tpi c_{\infty}} \\
  &=\frac{e^{\tpi c_1}-e^{\tpi (c_0+c_1)}-e^{\tpi (c_1 +c_{\infty})}+e^{\tpi (c_0+c_{\infty})}}{1-e^{\tpi c_1}}. 
\end{align*}
\begin{figure}
  \centering
  \scalebox{0.8}{\RWpictureJ}
  \caption{The twisted cycles $\ga_{10}$ and $\ga_{1\infty}^{\vee}$
    ($\bullet$ and $\circ$ are the intersection points.)}
  \label{fig:homology-intersection-10-1infty}
\end{figure}

\subsubsection{Proposition \ref{prop:homology-int-diag-0infty}}\label{subsubsec:homology-intersection-0-infty}
We consider the twisted cycles $\ga_{j0}$ and $\ga_{j\infty}$. 
If $j\neq k$, then it is obvious that $\ga_{j0}$ and $\ga_{k0}^{\vee}$ (resp. $\ga_{j\infty}$ and $\ga_{k\infty}^{\vee}$)
do not intersect topologically, and hence we obtain 
$I_h ([\ga_{j0}] ,[\ga_{k0}^{\vee}])=0$ (resp. $I_h ([\ga_{j\infty}] ,[\ga_{k\infty}^{\vee}])=0$). 

We now compute the self-intersection numbers. 
By Figures \ref{fig:homology-intersection-j0-j0} and \ref{fig:homology-intersection-jinfty-jinfty}, 
we obtain 
\begin{align*}
  I_h ([\ga_{j0}] ,[\ga_{j0}^{\vee}])
  &=\frac{1-e^{\tpi (c_0-c_1-\cdots -c_{j-1})}}{e^{\tpi c_j}-1} \cdot (-1) \cdot 1 \\
  &\qquad 
    +\frac{e^{\tpi c_j}-e^{\tpi (c_0-c_1-\cdots -c_{j-1}+c_j)}}{e^{\tpi c_j}-1} \cdot 1 
    \cdot e^{-\tpi c_0} e^{\tpi (c_1+\cdots +c_{j-1})} \\
  &=-\frac{(e^{\tpi c_0}-e^{\tpi (c_1+\cdots +c_{j-1})})(e^{\tpi c_0}-e^{\tpi (c_1+\cdots +c_j)})}
    {e^{\tpi (c_0+c_1+\cdots +c_{j-1})}(1-e^{\tpi c_j})}, 
\end{align*}
and 
\begin{align*}
  I_h ([\ga_{j\infty}] ,[\ga_{j\infty}^{\vee}])
  &=\frac{1-e^{\tpi (-c_{\infty}+c_1+\cdots +c_{j-1})}}{e^{\tpi c_j}-1} \cdot (-1) \cdot e^{\tpi c_j} \\
  &\qquad 
    +\frac{1-e^{\tpi (-c_{\infty}+c_1+\cdots +c_{j-1})}}{e^{\tpi c_j}-1}\cdot 1 \cdot 
    e^{\tpi (c_{\infty}-c_1-\cdots -c_{j-1})} \\
  &=-\frac{(e^{\tpi c_{\infty}}-e^{\tpi (c_1+\cdots +c_{j-1})})(e^{\tpi c_{\infty}}
    -e^{\tpi (c_1+\cdots +c_{j})})}{e^{\tpi (c_{\infty}+c_1+\cdots +c_{j-1})}(1-e^{\tpi c_j})}.
\end{align*}
\begin{figure}
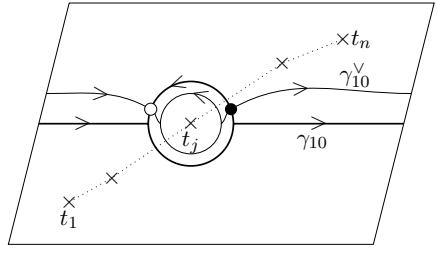

  \centering
  \scalebox{0.8}{\RWpictureK}  
  \caption{The twisted cycles $\ga_{j0}$ and $\ga_{j0}^{\vee}$
    ($\bullet$ and $\circ$ is the intersection points.)}
  \label{fig:homology-intersection-j0-j0}
\end{figure}
\begin{figure}
  \centering
  \scalebox{0.8}{\RWpictureL}
  \caption{The twisted cycles $\ga_{j\infty}$ and $\ga_{j\infty}^{\vee}$
    ($\bullet$ and $\circ$ are the intersection points.)}
  \label{fig:homology-intersection-jinfty-jinfty}
\end{figure}

\subsubsection{Proposition \ref{prop:homology-int-diag-1n}}
We compute the intersection numbers
$I_h ([\ga_{1j}] ,[\ga_{nk}^{\vee}])$ ($j=2,\dots ,n-1 ,0,\infty$; $k=2,\dots ,n-1 ,\infty,0$). 
It is clear that the intersection matrix becomes diagonal.
The non-zero intersection numbers are computed as follows: 
\begin{align*}
  I_h ([\ga_{1j}] ,[\ga_{nj}^{\vee}])
  &=-\frac{1}{e^{\tpi c_j}-1} \cdot 1 \cdot e^{-\tpi (c_{j+1}+\cdots +c_n)}
    =\frac{e^{\tpi (c_1+\cdots +c_j)}}{1-e^{\tpi c_j}} ,\\
  I_h ([\ga_{10}] ,[\ga_{n\infty}^{\vee}])
  &=1 \cdot 1 \cdot e^{\tpi (c_n+c_{\infty})}
    =e^{\tpi (c_n+c_{\infty})} ,\\ 
  I_h ([\ga_{1\infty}] ,[\ga_{n0}^{\vee}])
  &=1\cdot (-1) \cdot e^{-\tpi c_0}
    =-e^{-\tpi c_0}, 
\end{align*}
where $j=2,\dots ,n-1$.

\subsubsection{Some intersection matrices}\label{subsubsec:homology-int-matrices}
To obtain explicit formulas of connection matrices 
in \S \ref{subsubsec:connection+1} and \S \ref{subsubsec:connection+tau}, 
we need four intersection matrices 
$H_{10}=(I_h([\ga_{1j}] ,[\ga_{k0}^{\vee}]))$,
$H_{0n}=(I_h([\ga_{j0}] ,[\ga_{nk}^{\vee}]))$,
$H_{1\infty}=(I_h([\ga_{1j}] ,[\ga_{k\infty}^{\vee}]))$, and
$H_{0\infty}=(I_h([\ga_{j\infty}] ,[\ga_{nk}^{\vee}]))$.
The entries of them are computed similarly as above (in fact, some entries have been computed). 
We list them. 
In the following formulas, $j,k$ belong to $\{ 1,\dots ,n \}$ or its subset. 
\begin{itemize}
\item $H_{10}$: 
\begin{align*}
  &I_h([\ga_{1\infty}] ,[\ga_{k0}^{\vee}])=  
    \begin{cases}
      \frac{1-e^{-\tpi c_{\infty}}-e^{-\tpi c_0}+e^{\tpi (-c_0+c_1-c_{\infty})}}{1-e^{\tpi c_1}} & (k=1)\\
      -e^{-\tpi c_0} & (k=2,\dots ,n),
    \end{cases}
    \\
  &I_h([\ga_{10}] ,[\ga_{k0}^{\vee}])=
    \begin{cases}
      -\frac{(e^{\tpi c_0}-1)(e^{\tpi c_0}-e^{\tpi c_1})}{e^{\tpi c_0}(1-e^{\tpi c_1})}&(k=1) \\
      0&(k=2,\dots ,n),
    \end{cases}
    \\
  &I_h([\ga_{1j}],[\ga_{k0}^{\vee}])=  
    \begin{cases}
      -1  & (1\neq k <j) \\
      -\frac{1-e^{\tpi (-c_0+c_1+\dots +c_j)}}{1-e^{\tpi c_j}}& (j=k (>1)) \\
      0 & (k>j) \\
      \frac{e^{\tpi c_1}(1-e^{-\tpi c_0})}{1-e^{\tpi c_1}} & (k=1) .
    \end{cases}
\end{align*}
\item $H_{0n}$: 
\begin{align*}
    &I_h([\ga_{j0}] ,[\ga_{n\infty}^{\vee}])=
  \begin{cases}
    e^{\tpi (c_n +c_{\infty})}& (j=1,2,\dots ,n-1)\\
    \frac{e^{\tpi c_n}(1-e^{\tpi (c_0+c_n)}-e^{\tpi (c_n+c_{\infty})}
      +e^{\tpi (c_0+c_n+c_{\infty})})}{1-e^{\tpi c_n}}& (j=n),
  \end{cases} \\
  &I_h([\ga_{j0}] ,[\ga_{n0}^{\vee}])=
  \begin{cases}
    0& (j=1,2,\dots ,n-1)\\
    -\frac{(e^{\tpi (c_0+c_n)}-1)(e^{\tpi c_0}-1)}{e^{\tpi c_0}(1-e^{\tpi c_n})}& (j=n),
  \end{cases} \\
  &I_h([\ga_{j0}] ,[\ga_{nk}^{\vee}])=
    \begin{cases}
      -e^{\tpi c_0}  & (k<j\neq n) \\
      \frac{e^{\tpi c_j}(e^{\tpi c_0}-e^{\tpi (c_1+\cdots +c_{j-1})})}{1-e^{\tpi c_j}}& (k=j (< n)) \\
      0 & (k>j) \\
      \frac{1-e^{\tpi c_0}}{1-e^{\tpi c_n}}  &(j=n).
    \end{cases}
\end{align*}
\item $H_{1\infty}$:
\begin{align*}
  &I_h([\ga_{1\infty}] ,[\ga_{k\infty}^{\vee}])=
    \begin{cases}
       -\frac{(e^{\tpi c_{\infty}}-1)(e^{\tpi c_{\infty}}-e^{\tpi c_1})}{e^{\tpi c_{\infty}}(1-e^{\tpi c_1})} &(k=1) \\
       0 &(k=2,\dots ,n),
    \end{cases}
    \\
  &I_h([\ga_{10}] ,[\ga_{k\infty}^{\vee}])=
    \begin{cases}
      \frac{e^{\tpi c_1}-e^{\tpi (c_0+c_1)}-e^{\tpi (c_1 +c_{\infty})}+e^{\tpi (c_0+c_{\infty})}}{1-e^{\tpi c_1}} &(k=1)\\
      e^{\tpi (c_{\infty}-c_1-\cdots -c_{k-1})} & (k=2,\dots ,n),
    \end{cases}
    \\
  &I_h([\ga_{1j}] ,[\ga_{k\infty}^{\vee}])=
    \begin{cases}
      e^{\tpi (c_{\infty}-c_1-\dots -c_{k-1})}& (1\neq k <j) \\
      \frac{e^{\tpi (c_{\infty}-c_1-\cdots -c_{j-1})}-e^{\tpi c_j}}{1-e^{\tpi c_j}} & (j=k(> 1)) \\
      0 & (k>j) \\
      \frac{e^{\tpi c_1}(1-e^{\tpi c_{\infty}})}{1-e^{\tpi c_1}} & (k=1).
    \end{cases}
\end{align*}
\item $H_{\infty n}$: 
\begin{align*}
  &I_h([\ga_{j\infty}] ,[\ga_{n\infty}^{\vee}])=
  \begin{cases}
    0& (j=1,2,\dots ,n-1)\\
    -\frac{(e^{\tpi (c_{\infty}+c_n)}-1)(e^{\tpi c_{\infty}}-1)}{e^{\tpi c_{\infty}}(1-e^{\tpi c_n})}& (j=n),
  \end{cases} \\
  &I_h([\ga_{j\infty}] ,[\ga_{n0}^{\vee}])=
  \begin{cases}
    -e^{\tpi (-c_0+c_1+\cdots +c_{j-1})}& (j=1,2,\dots ,n-1)\\
    \frac{1-e^{-\tpi (c_{\infty}+c_n)}-e^{-\tpi (c_0+c_n)}+e^{-\tpi (c_0+c_n+c_{\infty})}}{1-e^{\tpi c_n}}& (j=n),
  \end{cases} \\
  &I_h([\ga_{j\infty}] ,[\ga_{nk}^{\vee}])=
    \begin{cases}
      e^{\tpi (c_1+\cdots +c_{j-1})}  & (k<j\neq n) \\
      -\frac{1-e^{\tpi (-c_{\infty}+c_1+\cdots +c_{j-1})}}{e^{\tpi (c_{j+1}+\cdots +c_n)}(1-e^{\tpi c_j})}& (k=j (<n)) \\
      0 & (k>j) \\
      \frac{1-e^{-\tpi c_{\infty}}}{e^{\tpi c_n}(1-e^{\tpi c_n})}  &(j=n).
    \end{cases}
\end{align*}
\end{itemize}

\subsection{Intersection numbers of twisted cocycles}\label{subsec:eval-cohomology-intersection}
Thanks to the formula (\ref{eq:intersecion-residue}), 
we can obtain the intersection number $I_c([\vph], [\vph'])$ by
\begin{itemize}
\item finding a Laurent series solution $f_l$ to the equation $\na f_l =\vph$ around $u=t_l$, and 
\item evaluating the residue $\Res_{u=t_l} (f_l \vph')$, 
\end{itemize}
for each $l=1,\dots ,n$. 
If it is clear that $f_l \vph'$ is holomorphic around $u=t_l$, 
an explicit form of $f_l$ is not needed. 

We first give some computation in a general setting. 
The logarithmic $1$-form $\om$ has a Laurent series expansion 
$\om /du =\sum_{m=-1}^{\infty} \al_m^{(l)} (u-t_l)^m$, where
\begin{align*}
  \al_{-1}^{(l)}=c_l ,\quad 
  \al_0^{(l)} =\tpi c_0 +\sum_{k\neq l} c_k \rho(t_l-t_k) ,\quad 
  \al_1^{(l)}=\frac{c_l}{3}\frac{\vth_1'''(0)}{\vth_1'(0)}
  +\sum_{k \neq l} c_k \rho'(t_l-t_k),\dots .
\end{align*}
We assume that $\vph$ has a Laurent expansion\footnote{
  Though $a_m$ should be written as, for example, $a_m^{(l)}$,
  we use this notation for simplicity. 
}
\begin{align*}
  \frac{\vph}{du}=\frac{a_{-2}}{(u-t_l)^2}+\frac{a_{-1}}{u-t_l} +a_0 +a_1 (u-t_l)+\cdots
\end{align*}
(in this paper, the differential forms that we consider have poles of order at most $2$). 
We set $f_l =\sum_{m=-1}^{\infty} b_m (u-t_l)^m$ and determine the coefficients $b_m$'s 
so that $\na f_l =\vph$ holds. 
By straightforward calculation, we have 
\begin{align*}
  b_{-1}
  &=\frac{a_{-2}}{\al_{-1}^{(l)}-1}
    =\frac{a_{-2}}{c_l-1},\\
  b_0
  &=\frac{a_{-1}-\al_0^{(l)} b_{-1}}{\al_{-1}^{(l)}}
    =\frac{a_{-1}}{c_l}-\frac{a_{-2} \al_0^{(l)}}{c_l(c_l-1)},\\
  b_1
  &=\frac{a_0-\al_1^{(l)} b_{-1}-\al_0^{(l)} b_0}{\al_{-1}^{(l)}+1}
    =\frac{a_0}{c_l+1}-\frac{a_{-1} \al_0^{(l)}}{c_l(c_l+1)}
    -\frac{a_{-2} (c_l \al_1^{(l)} -(\al_0^{(l)})^2)}{c_l(c_l-1)(c_l+1)},
\end{align*}
and so on. 
Especially, we can see $\ord_{t_l}(f_l)=\ord_{t_l}(\vph)+1$, and hence we obtain the following lemma. 
\begin{Lem}\label{lem:order-residue=0}
  If $\ord_{t_l}(\vph) +\ord_{t_l}(\vph')\geq -1$, then $\Res_{u=t_l}(f_l \vph')=0$. 
\end{Lem}

\subsubsection{Theorem \ref{th:cohomology-intersection-1}}
First, we compute the intersection number $I_c ([\psi_j],[\psi_k^{\vee}])$. 
The differential form $\psi_j=\frs (u-t_j ;\la)du$ has a pole at $u=t_j$ with residue $1$, and 
it is holomorphic around $u=t_l$ if $l\neq j$. 
Thus, 
Lemma \ref{lem:order-residue=0} shows 
$I_c ([\psi_j],[\psi_k^{\vee}])=0$ for $k\neq j$ and 
\begin{align*}
  I_c ([\psi_j],[\psi_j^{\vee}])
  =\tpi \Res_{u=t_j}(f_j \psi_j^{\vee})
  = \tpi \Res_{u=t_j}\left(
  \big( \frac{1}{c_j}+\cdots \big) \cdot \big( \frac{1}{u-t_j}+\cdots \big)du
  \right)
  =\frac{\tpi}{c_j}. 
\end{align*}

\subsubsection{Proposition \ref{prop:cohomology-intersection-2}}
Next, we consider $I_c ([\vph^{(pq)}_{j}], [\eta^{(pq)\vee}_{k}])$. 
Since we have 
\begin{align*}
  \ord_{t_l} (\vph^{(pq)}_q)
  \begin{cases}
    \geq 0 & (l\neq p) \\
    =-2 & (l=p),
  \end{cases}
  \quad 
  \ord_{t_l} (\vph^{(pq)}_j)
  \begin{cases}
    \geq 0 & (l\neq j) \\
    =-1 & (l=j)
  \end{cases}
  \quad (j\neq q),
\end{align*}
and 
\begin{align*}
  &\ord_{t_l} (\eta^{(pq)}_p)
  \begin{cases}
    \geq 0 & (l\neq p,q) \\
    =-1 & (l=p,q),
  \end{cases}
  \quad 
  \ord_{t_l} (\eta^{(pq)}_q)
  \begin{cases}
    \geq 0 & (l\neq q) \\
    =-1 & (l=q),
  \end{cases}\\
  &\ord_{t_l} (\eta^{(pq)}_k)
  \begin{cases}
    \geq 0 & (l\neq k,q) \\
    =1 & (l=p) \\
    =-1 & (l=k,q)
  \end{cases}
  \quad (k\neq p,q), 
\end{align*}
the property 
$I_c ([\vph^{(pq)}_{j}], [\eta^{(pq)\vee}_{k}])=0$ for $j\neq k$ 
follows from Lemma \ref{lem:order-residue=0}, except for $(j,k)=(q,p)$. 
To compute $I_c ([\vph^{(pq)}_{q}], [\eta^{(pq)\vee}_{p}])$, we solve $\na f_p =\vph^{(pq)}_{q}$ around $u=t_p$. 
By considering the termwise differentiation of the Laurent series of $\frs(u-t_p;\la)$, we have 
\begin{align*}
  \frac{\vph^{(pq)}_{q}}{du}=\frac{\pa \frs}{\pa u}(u-t_p ;\la)
  =\frac{-1}{(u-t_p)^2} +(\text{constant}) +\cdots ,
\end{align*}
and hence the Laurent expansion of $f_p$ has the form of 
\begin{align*}
  \frac{-1}{c_p -1} \cdot \frac{1}{u-t_p}
  -\frac{(-1) \cdot \al_0^{(p)}}{c_p(c_p-1)}  +\cdots
  =\frac{-1}{c_p -1} \cdot \frac{1}{u-t_p}
  +\frac{\al_0^{(p)}}{c_p(c_p-1)} +\cdots .
\end{align*}
On the other hand, the Laurent expansion of $\eta^{(pq)}_{p}$ is 
\begin{align*}
  \frac{\eta^{(pq)}_{p}}{du} 
  &=\frs(u-t_p;\la)+\frac{1}{\frs(t_p-t_q)}\Big( \frac{\al_0^{(p)}}{c_p} -\rho(-\la) \Big) \frs(u-t_q;\la) \\
  &=\Big( \frac{1}{u-t_p} +\rho(-\la) +\cdots \Big) 
    +\Big( \Big( \frac{\al_0^{(p)}}{c_p} -\rho(-\la)\Big) +\cdots \Big) 
  =\frac{1}{u-t_p} + \frac{\al_0^{(p)}}{c_p} +\cdots .
\end{align*}
Therefore, the intersection number $I_c ([\vph^{(pq)}_{q}], [\eta^{(pq)\vee}_{p}])$ is 
\begin{align*}
  I_c ([\vph^{(pq)}_{q}], [\eta^{(pq)\vee}_{p}])
  &=\tpi \Res_{u=t_p} (f_p \eta^{(pq)\vee}_{p}) \\
  &=\tpi \left( \frac{-1}{c_p -1} \cdot \Big(\frac{\al_0^{(p)}}{c_p}\Big)^{\vee} 
    +\frac{\al_0^{(p)}}{c_p(c_p-1)} \cdot 1\right) 
    =0 ,
\end{align*}
because of the property $(\al_0^{(p)})^{\vee}=-\al_0^{(p)}$. 

Let us compute $I_c ([\vph^{(pq)}_{j}], [\eta^{(pq)\vee}_{j}])$. 
The intersection numbers for $j\neq q$ are obtained by Theorem \ref{th:cohomology-intersection-1}, 
and $I_c ([\vph^{(pq)}_{q}], [\eta^{(pq)\vee}_{q}])$ is computed as follows: 
\begin{align*}
  &I_c ([\vph^{(pq)}_{q}], [\eta^{(pq)\vee}_{q}])
  =\tpi \Res_{u=t_p} (f_p \eta^{(pq)\vee}_{q}) \\
  &=\tpi \Res_{u=t_p} \left( \Big( \frac{-1}{c_p -1} \cdot \frac{1}{u-t_p}
    +\frac{\al_0^{(p)}}{c_p(c_p-1)} +\cdots \Big)
    \cdot \Big( \frs(t_p-t_q ;-\la) +\cdots \Big) du\right)\\
  &=-\tpi \cdot \frac{\frs(t_p-t_q;-\la)}{c_p-1}.
\end{align*}


\subsubsection{Theorem \ref{th:cohomology-intersection-0}}
Finally, we compute the intersection numbers when $\la =0$. 
Except for 
$I_c ([\vph_{ii}],[\vph_{ij}^{\vee}])$ ($j=0,1,\dots ,n$), 
the intersection numbers in Theorem \ref{th:cohomology-intersection-0} can be obtained 
similarly to the above discussion. 
By Lemma \ref{lem:order-residue=0}, we have 
$I_c ([\vph_{ii}],[\vph_{ij}^{\vee}]) =\tpi \Res_{u=t_i} (f_i \vph_{ij}^{\vee})$, 
where $f_i$ is a solution to $\na f_i =\vph_{ii}$ around $u=t_i$. 
Since we have 
\begin{align}\label{eq:Laurent-exp-vph_ii}
  \frac{\vph_{ii}}{du}=\rho'(u-t_i)
  =-\frac{1}{(u-t_i)^2} +\frac{\vth_1'''(0)}{3 \vth_1'(0)} +\cdots ,
\end{align}
the Laurent expansion of $f_i$ has the form of 
\begin{align*}
  &\frac{-1}{c_i -1} \cdot \frac{1}{u-t_i}
  -\frac{(-1) \cdot \al_0^{(i)}}{c_i(c_i-1)}
  +\Bigg( \frac{\frac{\vth_1'''(0)}{3 \vth_1'(0)}}{c_i+1} 
  -\frac{(-1)\cdot (c_i \al_1^{(i)} -(\al_0^{(i)})^2)}{c_i(c_i-1)(c_i+1)}\Bigg)(u-t_i)+\cdots \\
  &=\frac{-1}{c_i -1} \cdot \frac{1}{u-t_i}
  +\frac{\al_0^{(i)}}{c_i(c_i-1)}
  +\frac{1}{(c_i-1)(c_i+1)}\Bigg( (c_i -1)\frac{\vth_1'''(0)}{3 \vth_1'(0)}
  +\al_1^{(i)}-\frac{(\al_0^{(i)})^2}{c_i}\Bigg)(u-t_i)+\cdots .
\end{align*}
Therefore, by (\ref{eq:Laurent-exp-vph_ii}) and
\begin{align*}
  \vph_{i0} =du ,\qquad 
  \vph_{ij} =\left( -\frac{1}{u-t_i} +\rho(t_i-t_j) +\cdots \right) du \quad (j\neq 0,i), 
\end{align*}
the intersection numbers are computed as follows: 
\begin{align*}
  \frac{I_c ([\vph_{ii}],[\vph_{i0}^{\vee}])}{\tpi} 
  &=\frac{-1}{c_i -1},
  \\
  \frac{I_c ([\vph_{ii}],[\vph_{ii}^{\vee}])}{\tpi} 
  &=\frac{1}{(c_i-1)(c_i+1)}\Bigg( (c_i -1)\frac{\vth_1'''(0)}{3 \vth_1'(0)}
    +\al_1^{(i)}-\frac{(\al_0^{(i)})^2}{c_i}\Bigg) \cdot (-1)
    +\frac{-1}{c_i -1} \cdot \frac{\vth_1'''(0)}{3 \vth_1'(0)}
  \\
  &=\frac{1}{(c_i-1)(c_i+1)} \Bigg( 
    \frac{1}{c_i}\Big( \tpi c_0 +\sum_{k\neq i} c_k \rho(t_i-t_k) \Big)^2 
    -c_i\frac{\vth'''_{1}(0)}{\vth'_{1}(0)} 
    -\sum_{k\neq i} c_k \rho'(t_i-t_k)
    \Bigg), \\
  \frac{I_c ([\vph_{ii}],[\vph_{ij}^{\vee}])}{\tpi}
  &=\frac{\al_0^{(i)}}{c_i(c_i-1)} \cdot (-1)
    +\frac{-1}{c_i -1} \cdot \rho(t_i-t_j)
  \\
  &=-\frac{1}{c_i(c_i-1)}\Big( \tpi c_0 +\sum_{k\neq i} c_k \rho(t_i -t_k) 
    +c_i \rho(t_i -t_j) \Big), 
\end{align*}
where $j\neq 0,i$.

\begin{Ack}
  The author is grateful to Professors Toshiyuki Mano and Humihiko Watanabe 
  for their helpful advice. 
  This work was supported by JSPS KAKENHI Grant Number JP20K14276.
\end{Ack}

\end{document}